\documentclass{article}%
\usepackage{amssymb}
\usepackage{amsmath}
\usepackage{amsfonts}
\usepackage{graphicx}%
\providecommand{\U}[1]{\protect\rule{.1in}{.1in}}
\newtheorem{theorem}{Theorem}[section]
\newtheorem{acknowledgement}[theorem]{Acknowledgement}

\newtheorem{corollary}[theorem]{Corollary}

\newtheorem{example}[theorem]{Example}

\newtheorem{lemma}[theorem]{Lemma}

\newtheorem{proposition}[theorem]{Proposition}

\newenvironment{proof}[1][Proof]{\noindent\textbf{#1.} }{\ \rule{0.5em}{0.5em}}
\begin{document}

\title{The Multiplicative Ideal Theory of Leavitt Path Algebras}
\author{Kulumani M. Rangaswamy\\Department of Mathematics, \\University of Colorado at Colorado Springs\\Colorado Springs, Colorado 80918\\E-mail: krangasw@uccs.edu }
\date{}
\maketitle

\begin{abstract}
It is shown that every Leavitt path algebra $L$ of an arbitrary directed graph
$E$ over a field $K$ is an arithmetical ring, that is, the distributive law
$A\cap(B+C)=(A\cap B)+(A\cap C)$ holds for any three two-sided ideals of
\ $L$. It is also shown that $L$ is a multiplication ring, that is, given any
two ideals $A,B$ in $L$ with $A\subseteq B$, there is always an ideal $C$ such
that $A=BC$, an indication of a possible rich multiplicative ideal theory for
$L$. Existence and uniqueness of factorization of the ideals of $L$ as
products of special types of ideals such as prime, irreducible or primary
ideals is investigated.\ The irreducible ideals of $L$ turn out to be
precisely the primary ideals of $L$. It is shown that an ideal $I$ of $L$ is a
product of finitely many prime ideals if and only the graded part $gr(I)$ of
$I$ is a product of prime ideals and $I/gr(I)$ is finitely generated with a
generating set of cardinality no more than the number of distinct prime ideals
in the prime factorization of $gr(I)$. As an application, it is shown that if
$E$ is a finite graph, then every ideal of $L$ is a product of prime ideals.
The same conclusion holds if $L$ is two-sided artinian or two-sided
noetherian. Examples are constructed verifying whether some of the well-known
theorems in the ideal theory of commutative rings such as the Cohen's theorem
on prime ideals and the characterizing theorem on ZPI rings hold for Leavitt
path algebras.

\end{abstract}

\section{Introduction and Preliminaries.}

The ideals of a Leavitt path algebra $L$ of an arbitrary directed graph $E$
over a field $K$ seem to possess many desirable properties. For instance every
finitely generated ideal of $L$ is a principal ideal (see \cite{R2}) and that
the multiplication of ideals in $L$ is commutative (see the forthcoming book
\cite{AAS}). The last statement \ is somewhat surprising since Leavitt path
algebras are highly non-commutative. We shall give a direct proof of this
important result. Using these as a starting point, the main theme of this
paper is to explore the multiplicative ideal theory of Leavitt path algebras.
We first show (Theorem \ref{LPA is Arithmetic}) that every Leavitt path
algebra is an \textit{arithmetical ring}, that is, the distributive law
$A\cap(B+C)=(A\cap B)+(A\cap C)$ holds for any three ideals $A,B,C$ of $L$.
Arithmetical rings were introduced by L. Fuchs \cite{F-1} and commutative
arithmetical rings possess interesting properties allowing representations of
ideals as intersections/products of special types of ideals such as the prime,
irreducible, primary or primal ideals (see, e.g., \cite{F-2}, \cite{F-3},
\cite{F-4}, \cite{F-5}). Integral domains which are arithmetical rings have
come to be known as Pr\"{u}fer domains, a class of rings for which there are
over a hundred interesting characterizing properties.

One consequence of Theorem \ref{LPA is Arithmetic} is that the Chinese
remainder theorem holds in $L$. As another consequence, we show that $L$ is a
multiplication ring (Theorem \ref{LPAs are Multiplication rings}), a property,
first introduced by W. Krull, that is useful in factorizing ideals (see
\cite{LM}). Recall, a ring $R$ is a \textit{multiplication ring} if given any
two ideals $A,B$ of $R$ with $A\subseteq B$, there is an ideal $C$ such that
$A=BC$. Both the above-mentioned results point to a potential rich theory of
ideal factorizations in $L$. In this connection, factorization of graded
ideals in $L$ seems to influence that of non-graded ideals in $L$ and,
interestingly, this restricts the size of a generating set of the "non-graded
part" of the ideals. Specifically, we show that a non-graded ideal $I$ in a
Leavitt path algebra $L$ is a product of prime ideals if and only if its
graded part $gr(I)$ is a product of graded prime ideals and $I/gr(I)$ is
finitely generated with a generating set of cardinality at most the number of
distinct prime ideals in the prime factorization of $gr(I)$ (Theorem
\ref{gr(I) product of primes => I product of primes}). Products of primary
ideals and irreducible ideals are also considered. It is shown that, unlike in
the case of commutative rings where these two properties are independent, an
ideal $I$ of the Leavitt path algebra $L$ is a primary if and only if $I$ is
irreducible and this case $I=P^{n}$, a power of a prime ideal $P$. If further
$I$ is a graded ideal, then $I$ turns out to be a prime ideal. As an
application of the preceding results, we show that, if $E$ is a finite graph,
then every ideal in the Leavitt path algebra $L_{K}(E)$ is a product of prime
ideals. The same conclusion holds if $L_{K}(E)$ is two-sided artinian or
two-sided noetherian. Examples are constructed illustrating ideal
factorizations in $L$ and also examining whether some of the well-known
theorems in the ideal theory of commutative rings such as the Cohen's theorem
on prime ideals, the characterizing theorem on ZPI rings etc. hold for Leavitt
path algebras.

\section{Preliminaries}

For the general notation, terminology and results in Leavitt path algebras, we
refer to \cite{AAS}, \cite{R1} and \cite{T}. For basic results in associative
rings and modules, see \cite{Lam} and for commutative rings, we refer to
\cite{LM}. We give below an outline of some of the needed basic concepts and results.

A (directed) graph $E=(E^{0},E^{1},r,s)$ consists of two sets $E^{0}$ and
$E^{1}$ together with maps $r,s:E^{1}\rightarrow E^{0}$. The elements of
$E^{0}$ are called \textit{vertices} and the elements of $E^{1}$
\textit{edges}.

A vertex $v$ is called a \textit{sink} if it emits no edges and a vertex $v$
is called a \textit{regular} \textit{vertex} if it emits a non-empty finite
set of edges. An \textit{infinite emitter} is a vertex which emits infinitely
many edges. For each $e\in E^{1}$, we call $e^{\ast}$ a ghost edge. We let
$r(e^{\ast})$ denote $s(e)$, and we let $s(e^{\ast})$ denote $r(e)$.
A\textit{\ path} $\mu$ of length $n>0$ is a finite sequence of edges
$\mu=e_{1}e_{2}\cdot\cdot\cdot e_{n}$ with $r(e_{i})=s(e_{i+1})$ for all
$i=1,\cdot\cdot\cdot,n-1$. In this case $\mu^{\ast}=e_{n}^{\ast}\cdot
\cdot\cdot e_{2}^{\ast}e_{1}^{\ast}$ is the corresponding ghost path. A vertex
is considered a path of length $0$. The set of all vertices on the path $\mu$
is denoted by $\mu^{0}$.

A path $\mu$ $=e_{1}\dots e_{n}$ in $E$ is \textit{closed} if $r(e_{n}%
)=s(e_{1})$, in which case $\mu$ is said to be \textit{based at the vertex
}$s(e_{1})$. A closed path $\mu$ as above is called \textit{simple} provided
it does not pass through its base more than once, i.e., $s(e_{i})\neq
s(e_{1})$ for all $i=2,...,n$. The closed path $\mu$ is called a
\textit{cycle} if it does not pass through any of its vertices twice, that is,
if $s(e_{i})\neq s(e_{j})$ for every $i\neq j$.

An \textit{exit }for a path $\mu=e_{1}\dots e_{n}$ is an edge $e$ such that
$s(e)=s(e_{i})$ for some $i$ and $e\neq e_{i}$.

If there is a path from vertex $u$ to a vertex $v$, we write $u\geq v$. A
subset $D$ of vertices is said to be \textit{downward directed }\ if for any
$u,v\in D$, there exists a $w\in D$ such that $u\geq w$ and $v\geq w$. A
subset $H$ of $E^{0}$ is called \textit{hereditary} if, whenever $v\in H$ and
$w\in E^{0}$ satisfy $v\geq w$, then $w\in H$. A hereditary set is
\textit{saturated} if, for any regular vertex $v$, $r(s^{-1}(v))\subseteq H$
implies $v\in H$.

Given an arbitrary graph $E$ and a field $K$, the \textit{Leavitt path algebra
}$L_{K}(E)$ is defined to be the $K$-algebra generated by a set $\{v:v\in
E^{0}\}$ of pair-wise orthogonal idempotents together with a set of variables
$\{e,e^{\ast}:e\in E^{1}\}$ which satisfy the following conditions:

(1) \ $s(e)e=e=er(e)$ for all $e\in E^{1}$.

(2) $r(e)e^{\ast}=e^{\ast}=e^{\ast}s(e)$\ for all $e\in E^{1}$.

(3) (The "CK-1 relations") For all $e,f\in E^{1}$, $e^{\ast}e=r(e)$ and
$e^{\ast}f=0$ if $e\neq f$.

(4) (The "CK-2 relations") For every regular vertex $v\in E^{0}$,
\[
v=\sum_{e\in E^{1},s(e)=v}ee^{\ast}.
\]
Every Leavitt path algebra $L_{K}(E)$ is a $%
\mathbb{Z}
$\textit{-graded algebra}, namely, $L_{K}(E)=%
{\displaystyle\bigoplus\limits_{n\in\mathbb{Z}}}
L_{n}$ induced by defining, for all $v\in E^{0}$ and $e\in E^{1}$, $\deg
(v)=0$, $\deg(e)=1$, $\deg(e^{\ast})=-1$. Here the $L_{n}$ are abelian
subgroups satisfying $L_{m}L_{n}\subseteq L_{m+n}$ for all $m,n\in%
\mathbb{Z}
$. Further, for each $n\in%
\mathbb{Z}
$, the homogeneous component $L_{n}$ is given by
\[
L_{n}=\{%
{\textstyle\sum}
k_{i}\alpha_{i}\beta_{i}^{\ast}\in L:\text{ }|\alpha_{i}|-|\beta_{i}|=n\}.
\]
An ideal $I$ of $L_{K}(E)$ is said to be a \textit{graded ideal} if $I=$ $%
{\displaystyle\bigoplus\limits_{n\in\mathbb{Z}}}
(I\cap L_{n})$.

A \textit{breaking vertex }of a hereditary saturated subset $H$ is an infinite
emitter $w\in E^{0}\backslash H$ with the property that $0<|s^{-1}(w)\cap
r^{-1}(E^{0}\backslash H)|<\infty$. The set of all breaking vertices of $H$ is
denoted by $B_{H}$. For any $v\in B_{H}$, $v^{H}$ denotes the element
$v-\sum_{s(e)=v,r(e)\notin H}ee^{\ast}$. Given a hereditary saturated subset
$H$ and a subset $S\subseteq B_{H}$, $(H,S)$ is called an \textit{admissible
pair.} Given an admissible pair $(H,S)$, the ideal generated by $H\cup
\{v^{H}:v\in S\}$ is denoted by $I(H,S)$. It was shown in \cite{T} that the
graded ideals of $L_{K}(E)$ are precisely the ideals of the form $I(H,S)$ for
some admissible pair $(H,S)$. Moreover, $L_{K}(E)/I(H,S)\cong L_{K}%
(E\backslash(H,S))$. Here $E\backslash(H,S)$ is a \textit{Quotient graph of
}$E$ where\textit{\ }$(E\backslash(H,S))^{0}=(E^{0}\backslash H)\cup
\{v^{\prime}:v\in B_{H}\backslash S\}$ and $(E\backslash(H,S))^{1}=\{e\in
E^{1}:r(e)\notin H\}\cup\{e^{\prime}:e\in E^{1}$ with $r(e)\in B_{H}\backslash
S\}$ and $r,s$ are extended to $(E\backslash(H,S))^{0}$ by setting
$s(e^{\prime})=s(e)$ and $r(e^{\prime})=r(e)^{\prime}$.

Every graded ideal $I(H,S)$ in $L$ is isomorphic to a Leavitt path algebra of
some graph $F$ (see \cite{RT}) and hence contains \textit{local units}, that
is, to each $a\in I(H,S)$, there is an idempotent $u\in I(H,S)$ such that
$ua=a=au$.

We will also be using the fact that the Jacobson radical (and in particular,
the prime/Baer radical) of $L_{K}(E)$ is always zero (see \cite{AAS}).

Let $\Lambda$ be an arbitrary (possibly, infinite) index set. For any ring
$R$, we denote by $M_{\Lambda}(R)$ the ring of matrices over $R$ whose entries
are indexed by $\Lambda\times\Lambda$ and whose entries, except for possibly a
finite number, are all zero. It follows from the works in \cite{A}, \cite{AM}
that $M_{\Lambda}(R)$ is Morita equivalent to $R$.

Throughout this paper, $E$ will denote an arbitrary graph (with no restriction
on the number of vertices or on the number of edges emitted by each vertex)
and $K$ will denote an arbitrary field. For convenience in notation, we will
denote, most of the times, the Leavitt path algebra $L_{K}(E)$ by $L$.
Finally, we write ideals to denote two-sided ideals and, by a product of
ideals, we shall always mean a product of finitely many ideals. Also $<a>$
denotes the ideal generated by the element $a$.

The following two results will be used in our investigation.

\begin{theorem}
\label{Ideal generated by c^0- c cycle withour exts}(Proposition 3.5,
\cite{AAPS}) Suppose $E$ is an arbitrary graph and $\{c_{t}:t\in T\}$ is the
set of all cycles without exits in $E$. Let $M$ be the ideal of $L_{K}(E)$
generated by the vertices in all the cycles $c_{t}$ with $t\in T$. Then $M$ is
a ring direct sum $M=%
{\displaystyle\bigoplus\limits_{t\in T}}
M_{t}$ where, for each $t$, $M_{t}$ is the ideal generated by the vertices on
the cycle $c_{t}$ and $M_{t}\cong M_{\Lambda_{t}}(K[x,x^{-1}])$ with
$\Lambda_{t}$ a suitable index set depending on $t$.
\end{theorem}

\begin{theorem}
\label{R-2}(Theorem 4, \cite{R2}) Let $I$ be a non-graded ideal of
$L=L_{K}(E)$ with $H=I\cap E^{0}$ and $S=\{u\in B_{H}:u^{H}\in I\}$. Then

(a) $I=I(H,S)+%
{\displaystyle\sum\limits_{t\in T}}
<f_{t}(c_{t})>$ where $T$ is some index set, for each $t\in T$, $c_{t}$ is a
cycle without exits in $E\backslash(H,S)$, $c_{t}^{0}\cap c_{s}^{0}=\emptyset$
for $t\neq s$ and $f_{t}(x)\in K[x]$ is a polynomial of the smallest degree
such that $f_{t}(c_{t})\in I$.

(b) \cite{R2} $I(H,S)$\ is the largest graded ideal inside $I$ and if $I$\ is
a prime ideal, then $I(H,S)$\ is also a prime ideal.
\end{theorem}

Note that if $I$ is the non-graded ideal considered in Theorem \ref{R-2}, then
$I/gr(I)=%
{\displaystyle\bigoplus\limits_{t\in T}}
<f_{t}(c_{t})>$. This is because, in $L/I(H,S)$, $<f_{t}(c_{t})>\subseteq
M_{t}$, the ideal generated by the vertices on the cycle $c_{t}$ and by
Theorem \ref{Ideal generated by c^0- c cycle withour exts}, $%
{\displaystyle\sum}
M_{t}=%
{\displaystyle\bigoplus}
M_{t}$.

\textbf{Notation}: For convenience, $I(H,S)$ will\ also be denoted by $gr(I)$
and we shall call it the graded part of $I$.

\section{Some Properties of the Ideals in Leavitt Path Algebras}

Let $E$ be an arbitrary graph. We begin by describing some special features of
the graded ideals in the Leavitt path algebra $L=L_{K}(E)$. We first show that
an ideal $A$ of $L$ is graded if and only if $A\cap B=AB$ for any ideal $B$ of
$L$. It then follows easily that if $A$ is a graded ideal, then $AB=BA$ for
any ideal $B$ of $L$. Actually, the commutativity property $AB=BA$ holds even
for non-graded ideals. This interesting result is stated and proved in
\cite{AAS} by using a deep structure theorem on ideals of $L_{K}(E)$ and using
a lattice isomorphism with the ideal lattice of $L_{K}(E)$. Since the book
\cite{AAS} is being written up and is yet to be published, and since this
result is very relevant to the multiplicative ideal theory of Leavitt path
algebras that we are investigating, we give below a direct proof of this
important result using the ideas considered here in this paper.

A very useful observation is that if $c$ is a cycle based at a vertex $v$ in
$E$, then $vLv\cong K[x,x^{-1}]$ under an isomorphism mapping $v$ to $1$, $c$
to $x$ and $c^{\ast}$ to $x^{-1}$.

Our first Lemma points out some special characteristics of graded ideals of
$L$.

\begin{lemma}
\label{Product = Intersection} \ (i) An ideal $A$ of $L$ is a graded ideal if
and only if for any ideal $B$ of $L$, $AB=A\cap B$ and $BA=B\cap A$. Thus a
graded ideal $A$ satisfies $A^{2}=A$ and $AB=BA$ for all ideals $B$.

(ii) Suppose $A$ is a graded ideal. Then $A=A_{1}\cdot\cdot\cdot A_{m}$ is a
product of ideals if and only if $A=A_{1}\cap\cdot\cdot\cdot\cap A_{m}$ is an
intersection of the ideals $A_{i}$. In this case, $A=%
{\displaystyle\bigcap\limits_{i=1}^{m}}
gr(A_{i})=$ $%
{\displaystyle\prod\limits_{i=1}^{m}}
gr(A_{i})$.

(iii) If $A_{1},\cdot\cdot\cdot,A_{m}$ are graded ideals of $L$, then $%
{\displaystyle\prod\limits_{i=1}^{m}}
A_{i}=%
{\displaystyle\bigcap\limits_{i=1}^{m}}
A_{i}$.
\end{lemma}

\begin{proof}
(i) Suppose $A$ is a graded ideal of $L$. Clearly $AB\subseteq A\cap B$. To
prove the reverse inclusion, let $x\in A\cap B$. Since $A$ is a graded ideal,
there is a local unit $u\in A$ satisfying $x=ux=xu$. Then $x=ux\in AB$. So
$A\cap B=AB$. Similarly, $B\cap A=BA$. Hence $AB=BA$. In particular,
$A^{2}=A\cap A=A$.

Conversely, suppose $A\cap B=AB$ for all ideals $B$ in $L$. Suppose, on the
contrary, $A$ is not graded. By Theorem \ref{R-2}, $A=I(H,S)+%
{\displaystyle\sum\limits_{i\in X}}
<f_{i}(c_{i})>$ where each $c_{i}$ is a cycle without exits in $E\backslash
(H,S)$ and $f_{i}(x)\in K[x]$. For convenience, write $N=%
{\displaystyle\sum\limits_{i\in X}}
<f_{i}(c_{i})>$ and $gr(A)=I(H,S)$. Then $AN=(gr(A)+N)N=gr(A)N+N^{2}%
=(gr(A)\cap N)+N^{2}$. Since $AN=A\cap N=N$, $(gr(A)\cap N)+N^{2}=N$ and so
$[N^{2}+gr(A)]/gr(A)=(N+gr(A))/gr(A)$. This means, in the Leavitt path algebra
$\bar{L}=L/gr(A)$,%
\[%
{\displaystyle\bigoplus\limits_{i}}
[<f_{i}(c_{i})>]^{2}=%
{\displaystyle\bigoplus\limits_{i}}
[<f_{i}(c_{i})>].
\]
Since $%
{\displaystyle\bigoplus}
$ is a ring direct sum, we obtain $[<f_{i}(c_{i})>]^{2}=<f_{i}(c_{i})>$. Since
$f_{i}(c_{i})\in v_{i}\bar{L}v_{i}\cong K[x,x^{-1}]$, we get $<f_{i}%
(x)>^{2}=<f_{i}(x)>$ in the integral domain $K[x,x^{-1}]$, a contradiction.
This contradiction shows that $A$ is a graded ideal.

(ii) Suppose $A=A_{1}\cdot\cdot\cdot A_{m}$ is a product of ideals. Now,
$\left(  [A_{1}\cap\cdot\cdot\cdot\cap A_{m}]/A\right)  ^{m}=0$ in $L/A$. On
the other hand, if we write the graded ideal $A$ as $A=I(H,S)$, then $L/A\cong
L_{K}(E\backslash(H,S))$ and so $L/A$ contains no non-zero nilpotent ideals.
Consequently, $[A_{1}\cap\cdot\cdot\cdot\cap A_{m}]/A=0$ or $A=A_{1}\cap
\cdot\cdot\cdot\cap A_{m}$.

Conversely, suppose $A=A_{1}\cap\cdot\cdot\cdot\cap A_{m}$. Since $A$ is
graded, $A=A^{m}$ by (i), and so we get
\[
A=A^{m}\subseteq A_{1}\cdot\cdot\cdot A_{m}\subseteq A_{1}\cap\cdot\cdot
\cdot\cap A_{m}=A\text{. }%
\]
Thus, $A=A_{1}\cdot\cdot\cdot A_{m}$.

Finally, if $A=%
{\displaystyle\bigcap\limits_{i=1}^{m}}
A_{i}$, clearly $%
{\displaystyle\bigcap\limits_{i=1}^{m}}
gr(A_{i})\subseteq A$. On the other hand, since $A$ is graded, $A\subseteq
gr(A_{i})$ for all $i$ and hence $A\subseteq%
{\displaystyle\bigcap\limits_{i=1}^{m}}
gr(A_{i})$. Thus $A=%
{\displaystyle\bigcap\limits_{i=1}^{m}}
gr(A_{i})$ which is equal to $%
{\displaystyle\prod\limits_{i=1}^{m}}
gr(A_{i})$, by (i).

(iii) Now $%
{\displaystyle\bigcap\limits_{i=1}^{m}}
A_{i}$ is a graded ideal and so (iii) follows from (ii).
\end{proof}

\begin{lemma}
\label{graded multiplicatin ring} Suppose $A,B$ are ideals of $L$ with
$A\subseteq B$. If $A\subseteq gr(B)$, then $AB=A=BA$. \ In particular, if $A$
or $B$ is a graded ideal, then $AB=A=BA$.
\end{lemma}

\begin{proof}
If $A\subseteq gr(B)$, then $A=A\cap gr(B)=Agr(B)\subseteq AB\subseteq A$. So
$A=AB$. Similarly, $BA=A$. If $A$ is graded, then $A\subseteq gr(B)$. If $B$
is graded, again $A\subseteq B=gr(B)$. From the first part, we then conclude
that $A=AB=BA$.
\end{proof}

Next we give a direct proof of the more general statement that the ideal
multiplication in a Leavitt path algebra is indeed commutative.

We start with the following simple Lemma. In its proof we shall be using the
fact that whenever $p^{\ast}q\neq0$, where $p,q$ are paths in $E$, then the
CK-1 relation in $L_{K}(E)$ implies that either $p=qr$ or $q=ps$ where $r,s$
are suitable paths in $E$.

\begin{lemma}
\label{Pricipal ideals commute}Let $E$ be an arbitrary graph. If $c$ is a
cycle without exits based at a vertex $v$ in $E$ and $f(x),g(x)\in K[x]$,
then
\[
<f(c)><g(c)>=<f(c)g(c)>=<g(c)><f(c)>.
\]

\end{lemma}

\begin{proof}
Now a typical element of $<f(c)><g(c)>$ is a $K$-linear combination of
non-zero terms of the form $\alpha\beta^{\ast}f(c)\gamma\delta^{\ast}%
g(c)\mu\nu^{\ast}$ where $\alpha,\beta,\gamma,\delta,\mu,\nu$ are paths in
$E$. Here $s(\gamma)=v=s(\delta)$ and $r(\gamma)=r(\delta)$. Since $c$ is a
cycle without exits, $\gamma\delta^{\ast}$ simplifies to an integer power of
$c$ or $c^{\ast}$ which we denote by $c^{\epsilon}$. Then $\alpha\beta^{\ast
}f(c)\gamma\delta^{\ast}g(c)\mu\nu^{\ast}=\alpha\beta^{\ast}c^{\epsilon
}f(c)g(c)\mu\nu^{\ast}\in<f(c)g(c)>$. On the other hand, $<f(c)g(c)>\subseteq
<f(c)><g(c)>$ and so we get
\[
<f(c)><g(c)>=<f(c)g(c)>=<g(c)f(c)>.
\]
In a similar fashion, we can show that $<g(c)><f(c)>=<g(c)f(c)>$.
\end{proof}

\begin{theorem}
\label{Ideals in LPA commute}(\cite{AAS}) Let $E$ be an arbitrary graph. Then
for any two ideals $A,B$ in $L=L_{K}(E)$, we have $AB=BA$.
\end{theorem}

\begin{proof}
In view of Lemma \ref{Product = Intersection} (i), we may assume that both $A$
and $B$ are non-graded ideals. We distinguish two cases.

Case 1: Suppose $A$ or $B$ is contained in $gr(A)+gr(B)$, say $A\subseteq
gr(A)+gr(B)$. \ By modular law, $A=gr(A)+(A\cap gr(B))$. Then, by using Lemma
\ref{Product = Intersection}(i) and Lemma \ref{graded multiplicatin ring}, we
get
\[
AB=gr(A)B+[A\cap gr(B)]B=(gr(A)\cap B)+(A\cap gr(B)).
\]
Similarly,%
\[
BA=Bgr(A)+B[A\cap gr(B)]=(B\cap gr(A))+(A\cap gr(B).
\]
Thus $AB=BA$.

Case 2: Suppose $A\nsubseteqq gr(A)+gr(B)$ and $B\nsubseteqq gr(A)+gr(B)$. We
write $gr(A)+gr(B)=I(H,S)$ where $H=(gr(A)+gr(B))\cap E^{0}$ and $S=\{u\in
B_{H}:u^{H}\in gr(A)+gr(B)\}$. By Theorem \ref{R-2}, we can write
\[
A=gr(A)+%
{\displaystyle\sum\limits_{j\in X}}
<f_{j}(c_{j})>\ \text{and }B=gr(B)+%
{\displaystyle\sum\limits_{k\in Y}}
<g_{k}(c_{k})>
\]
where $X,Y$ are some index sets, for each $j\in X$ and $k\in Y$,
$f_{j}(x),g_{k}(x)\in K[x]$, $c_{j}$ and $c_{k}$ are cycles without exits in
$E^{0}\backslash gr(A)$ and in $E^{0}\backslash gr(B)$, respectively. In
$L/I(H,S)\cong L_{K}(E\backslash(H,S))$, let $M$ denote the ideal of generated
by the vertices in all the cycles $c_{t}$ ($t\in T$) without exits in
$E\backslash(H,S)$. By Theorem
\ref{Ideal generated by c^0- c cycle withour exts}, $M$ is a ring direct sum
$M=%
{\displaystyle\bigoplus\limits_{t\in T}}
M_{t}$ where $M_{t}$ is the ideal generated by the vertices on the cycle
$c_{t}$ in $L/I(H,S)$. Now $\bar{A}=(A+I(H,S))/I(H,S)$ (and likewise $\bar
{B}=(B+I(H,S))/I(H,S)$) is an epimorphic image of $A/gr(A)$ (of $B/gr(B)$) and
so $\bar{A},\bar{B}\subseteq M$. Consequently, we can write $\bar{A}=%
{\displaystyle\bigoplus\limits_{r\in X^{\prime}\subseteq T}}
<f_{r}(c_{r})>$ and $\bar{B}=%
{\displaystyle\bigoplus\limits_{s\in Y^{\prime}\subseteq T}}
<g_{s}(c_{s})>$ where $c_{r},c_{s}$ are cycles without exits in $E\backslash
(H,S)$. Since the product $M_{r}M_{s}=0$ for all $r\neq s$ in $T$, we have,
using Lemma \ref{Pricipal ideals commute},
\[
\bar{A}\bar{B}=%
{\displaystyle\bigoplus\limits_{k\in X^{\prime}\cap Y^{\prime}}}
<f_{k}(c_{k})><g_{k}(c_{k})>=%
{\displaystyle\bigoplus\limits_{k\in X^{\prime}\cap Y^{\prime}}}
<g_{k}(c_{k})><f_{k}(c_{k})>=\bar{B}\bar{A}.
\]
Then $AB=BA+gr(A)+gr(B)$. Now $A\cap B$ contains both $AB,BA$ and so, using
modular law,
\begin{align*}
AB  &  =(A\cap B)\cap AB=BA+(A\cap B)\cap\lbrack gr(A)+gr(B)]\\
&  =BA\text{, as }BA\text{ contains the second term.}%
\end{align*}

\end{proof}

\section{Leavitt Path Algebras as Arithmetical Rings and Multiplication Rings}

The main theorem of this section shows that the ideals of a Leavitt path
algebra $L$ form a distributive lattice. As a consequence, the Chinese
remainder theorem holds in $L$. We begin with two propositions which are
useful in the proof of the main theorem (Theorem \ref{LPA is Arithmetic}) and
which show how things work out nicely for graded ideals. Using Theorem
\ref{LPA is Arithmetic}, we then show that every Leavitt path algebra $L$ is a
multiplication ring, a useful property in the ideal theory of rings.

\begin{proposition}
\label{Distributive Lattice} Let $E$ be an arbitrary graph. If $A,B,C$ are
ideals of $L$ and if one of them is a graded ideal, then we have
\[
A\cap(B+C)=(A\cap B)+(A\cap C).
\]

\end{proposition}

\begin{proof}
Suppose $A$ is a graded ideal. By Lemma \ref{Product = Intersection}, we then
obtain
\[
A\cap(B+C)=A(B+C)=AB+AC=(A\cap B)+(A\cap C).\qquad\qquad
\]
Next, suppose one of $B$ and $C$, say $B$, is a graded ideal. We need only to
show that $A\cap(B+C)\subseteq(A\cap B)+(A\cap C)$. Let $a=b+c\in A\cap(B+C)$.
Since $B$ is graded, $b=ubu$ for some local unit $u\in B$. So
$b=ubu=uau-ucu\in(A\cap B)+(B\cap C)$. Then $c=a-b\in C\cap\lbrack A+(A\cap
B)+(B\cap C)]=C\cap\lbrack A+(B\cap C)]=(C\cap A)+(B\cap C)$, by modular law.
It then follows that
\begin{align*}
a  &  =b+c\in A\cap\lbrack(A\cap B)+(B\cap C)+(C\cap A)+(B\cap C)]\\
&  =(A\cap B)+(A\cap C)+(A\cap B\cap C)\\
&  =(A\cap B)+(A\cap C).
\end{align*}
Similar argument holds when $C$ is a graded ideal.
\end{proof}

\begin{proposition}
\label{Mod graded B+C} Let $A,B,C$ be non-graded ideals of $L$. If one of
$A,B,C$ is contained in $gr(B)+gr(C)$, then $A\cap(B+C)=(A\cap B)+(A\cap C)$.
\end{proposition}

\begin{proof}
Suppose $A\subseteq gr(B)+gr(C)$. Then
\begin{align*}
A\cap(B+C)  &  =A\cap\lbrack gr(B)+gr(C)]\cap(B+C)\\
&  =A\cap\lbrack gr(B)+gr(C)]\\
&  =(A\cap gr(B))+(A\cap gr(C))\text{, by Proposition
\ref{Distributive Lattice}}\\
&  \subseteq(A\cap B)+(A\cap C).
\end{align*}
Suppose one of $B$ and $C$, say, $B\subseteq gr(B)+gr(C)$. Then
$B=gr(B)+(B\cap gr(C))$, by modular law. Consequently,
\begin{align*}
A\cap(B+C)  &  =A\cap([gr(B)+(B\cap gr(C))]+C)\\
&  =A\cap(gr(B)+C)\\
&  =(A\cap gr(B))+(A\cap C)\text{, by Proposition \ref{Distributive Lattice}%
}\\
&  \subseteq(A\cap B)+(A\cap C).
\end{align*}
Similar argument works when $C\subseteq gr(B)+gr(C)$.
\end{proof}

We are now ready to show that every Leavitt path algebra is an arithmetical ring.

\begin{theorem}
\label{LPA is Arithmetic} The ideals of any Leavitt path algebra form a
distributive lattice. Specifically, if $A,B,C$ are any three ideals of a
Leavitt path algebra $L$ of an arbitrary graph $E$, then
\[
A\cap(B+C)=(A\cap B)+(A\cap C).
\]

\end{theorem}

\begin{proof}
In view of Propositions \ref{Distributive Lattice} and \ref{Mod graded B+C},
we may assume that $A,B,C$ are all non-graded ideals such that none of $A,B$
or $C$ is contained in $gr(B)+gr(C)$. By Theorem \ref{R-2},
\[
B=I(H_{1},S_{1})+%
{\displaystyle\sum\limits_{i\in X}}
<f_{i}(c_{i})>\text{and }C=I(H_{2},S_{2})+%
{\displaystyle\sum\limits_{j\in Y}}
<g_{j}(c_{j})>
\]
where $X,Y$ are some index sets, $I(H_{1},S_{1})=gr(B)$, $I(H_{2}%
,S_{2})=gr(C)$, for all $i\in X$ and $j\in Y$, $f_{i}(x),g_{j}(x)\in K[x]$ and
the $c_{i},c_{j}$ are cycles without exits in $E\backslash(H_{1},S_{1})$ and
$E\backslash(H_{2},S_{2})$ respectively. Let $I(H,S)$ denote the graded ideal
$gr(B)+gr(C)$. Consider $L/I(H,S)$ which we identify with $L_{K}%
(E\backslash(H,S))$. Let $\bar{A},\bar{B},\bar{C}$ denote the images\ in
$L/I(H,S)$ of $A,B,C$ respectively. Let $M$ be the ideal of $L/I(H,S)$
generated by the vertices in all the cycles without exits in $E\backslash
(H,S)$. Since $\bar{B}=[B+I(H,S)]/I(H,S)$ (and likewise, $\bar{C}%
=[C+I(H,S)]/I(H,S)$) is a non-zero homomorphic image of $B/I(H_{1},S_{1})$
(respectively, of $C/I(H_{2},S_{2})$), $\bar{B},\bar{C}\subseteq M$.

If $\bar{A}\cap(\bar{B}+\bar{C})=0$, then $A\cap(B+C)\subseteq gr(B)+gr(C)$
and we get,%
\begin{align*}
A\cap(B+C)  &  =A\cap\lbrack A\cap(B+C)]\subseteq A\cap\lbrack gr(B)+gr(C)]\\
&  =(A\cap gr(B))+(A\cap gr(C))\text{, by Proposition
\ref{Distributive Lattice}}\\
&  \subseteq(A\cap B)+(A\cap C)\text{.}%
\end{align*}

Suppose $\bar{A}\cap(\bar{B}+\bar{C})\neq0$, so $A^{\prime}=\bar{A}\cap
M\neq0$. Now, by Theorem \ref{Ideal generated by c^0- c cycle withour exts},
$M$ is a ring direct sum of matrix rings of the form $M_{\Lambda}%
(K[x,x^{-1}])$ whose ideal lattice is distributive (as $M_{\Lambda}%
(K[x,x^{-1}])$ is Morita equivalent to the principal ideal domain
$K[x,x^{-1}]$). Hence the ideal lattice of $M$ is also distributive and we
obtain
\[
\bar{A}\cap(\bar{B}+\bar{C})=A^{\prime}\cap(\bar{B}+\bar{C})=(A^{\prime}%
\cap\bar{B})+(A^{\prime}\cap\bar{C})=(\bar{A}\cap\bar{B})+(\bar{A}\cap\bar
{C})\text{.}%
\]
Then
\[
A\cap(B+C)\subseteq(A\cap B)+(A\cap C)+gr(B)+gr(C).\qquad\qquad(\ast)
\]
Intersecting with $A$ on both sides of $(\ast)$ and using modular law and
Proposition \ref{Distributive Lattice}, we obtain%
\begin{align*}
A\cap(B+C)  &  \subseteq(A\cap B)+(A\cap C)+A\cap\lbrack gr(B)+gr(C)]\\
&  =(A\cap B)+(A\cap C)+(A\cap(gr(B))+(A\cap gr(C))\\
&  =(A\cap B)+(A\cap C)\text{.}%
\end{align*}
This proves that the ideal lattice of $L$ is distributive.
\end{proof}

\textbf{Remark}: As an interesting consequence of Theorem
\ref{LPA is Arithmetic}, one can show, using well-known arguments (see e.g.
Theorem 18, ch IV, \cite{ZS}) that the Chinese remainder theorem holds in
Leavitt path algebras.

We next use Theorem \ref{LPA is Arithmetic}, to show that every Leavitt path
algebra is a multiplication ring, a useful property in the multiplicative
ideal theory of Leavitt path algebras.

\begin{theorem}
\label{LPAs are Multiplication rings}The Leavitt path algebra $L=L_{K}(E)$ of
an arbitrary graph $E$ is a multiplication ring, that is, for any two ideals
$A,B$ of $L$ with $A\subseteq B$, there is an ideal $C$ of $L$, such that
$A=BC=CB$.
\end{theorem}

\begin{proof}
If $A=B$, then, as $L$ is a ring with local units, $A=B=BL$. So assume that
$A\subsetneqq B$. If $A$ or $B$ is a graded ideal or if $A\subseteq gr(B)$,
then, as shown in Lemma \ref{graded multiplicatin ring}, $C=A$ satisfies
$A=BC=CB$. So\ we may assume that both $A$ and $B$ are non-graded ideals with
$A\subsetneqq B$ but $A\nsubseteq gr(B)$. By Theorem \ref{R-2}, $A=I(H_{1}%
,S_{1})+%
{\displaystyle\sum\limits_{i\in X}}
<f_{i}(c_{i})>$ and $B=I(H_{2},S_{2})+%
{\displaystyle\sum\limits_{j\in Y}}
<g_{j}(c_{j})>$ where $I(H_{1},S_{1})=gr(A)\subseteq I(H_{2},S_{2})=gr(B)$,
the $c_{i}$ and $c_{j}$ are cycles without exits in $E\backslash H_{1}$ and in
$E\backslash H_{2}$ respectively, based at vertices $v_{i},v_{j}$ and
$f_{i}(x),g_{j}(x)\in K[x]$. Let $T=\{i\in X:v_{i}\in gr(B)\}$ so that $%
{\displaystyle\sum\limits_{i\in T}}
<f_{i}(c_{i})>\subseteq gr(B)$. Moding out $gr(A)$ and denoting $\bar
{A}=A/gr(A)$ and $\bar{B}=B/gr(A)$, we have in $\bar{L}=L/gr(A)\cong
L_{K}(E\backslash(H_{1},S_{1})$,
\[
\bar{A}=%
{\displaystyle\sum\limits_{i\in X}}
<f_{i}(c_{i})>\subseteq\bar{B}=[gr(B)/gr(A)]+%
{\displaystyle\sum\limits_{j\in Y}}
<g_{j}(c_{j})>.
\]
Now $B/gr(B)=%
{\displaystyle\bigoplus\limits_{j\in Y}}
<g_{j}(c_{j})>$ is a ring direct sum of ideals and the ideal
\[
\lbrack A+gr(B)]/gr(B)=%
{\displaystyle\bigoplus\limits_{j\in X\backslash T}}
<f_{j}(c_{j})>\subseteq B/gr(B).
\]
It is then easy to see that, for each $j\in X\backslash T$, $<f_{j}%
(c_{j})>\subseteq<g_{j}(c_{j})>$. From Theorem \ref{R-2}(a), $g_{j}(x)$ is a
polynomial of the smallest degree in $K[x]$ such that $g_{j}(c_{j})\in B$.
Consequently, $g_{j}(x)$ must be a divisor of $f_{j}(x)$ in $K[x]$. Hence, for
each $j\in X\backslash T$, $f_{j}(c_{j})=g_{j}(c_{j})f_{j}^{\prime}(c_{j})$
where $f_{j}^{\prime}(x)\in K[x]$ satisfies $f_{j}(x)=g_{j}(x)f_{j}^{\prime
}(x)$.

Our goal is to show that $A=BC$, where
\[
C=gr(A)+%
{\displaystyle\sum\limits_{i\in T}}
<f_{i}(c_{i})>+%
{\displaystyle\sum\limits_{j\in X\backslash T}}
<f_{j}^{\prime}(c_{j})>\text{.}%
\]
We first prove the following claim where $gr(\bar{B})=$ $gr(B)/gr(A)$.

\textbf{Claim}: In $\bar{L}=L/gr(A)$,
\[
gr(\bar{B})\cap\lbrack%
{\displaystyle\sum\limits_{j\in X\backslash T}}
<h_{j}(c_{j})>]=0\qquad\qquad\qquad(\ast)
\]
where, for each $j\in X\backslash T$, $h_{j}(x)$ is any arbitrarily chosen
polynomial belonging to $K[x]$.

\textbf{Proof of the claim:} We identify $\bar{L}$ with $L_{K}(E\backslash
(H_{1},S_{1}))$. By Theorem \ref{LPA is Arithmetic}, distributive law holds
and so we have
\[
gr(\bar{B})\cap\lbrack%
{\displaystyle\sum\limits_{j\in X\backslash T}}
<h_{j}(c_{j})>]=%
{\displaystyle\sum\limits_{j\in X\backslash T}}
\left[  gr(\bar{B})\cap<h_{j}(c_{j})>\right]  .
\]
Now, by Lemma \ref{Product = Intersection}, $gr(\bar{B})\cap<h_{\lambda
}(c_{\lambda})>=gr(\bar{B})\cdot<h_{\lambda}(c_{\lambda})>$ where $\cdot$
denotes the product. Hence it is enough to show that $gr(\bar{B}%
)\cdot<h_{\lambda}(c_{\lambda})>=0$, for any given $\lambda\in X\backslash T$.
For convenience in writing, denote the graded ideal $gr(\bar{B})$ by $I(H,S)$
where $H=H_{2}\backslash H_{1}\subseteq E\backslash(H_{1},S_{1})$. By Lemma
5.6 in \cite{T}, the elements of the graded ideal $gr(\bar{B})$ are of the
form $%
{\displaystyle\sum\limits_{j=1}^{m}}
l_{j}\alpha_{j}\beta_{j}^{\ast}+%
{\displaystyle\sum\limits_{i=1}^{n}}
t_{i}\gamma_{i}v_{i}^{H}\delta_{i}^{\ast}$ where $l_{j},t_{i}\in K$,
$r(\alpha_{j})=r(\beta_{j})\in H$ and $r(\gamma_{i})=r(\delta_{i})=u_{i}\in S$.

Let $x\in I(H,S)\cdot<h_{\lambda}(c_{\lambda})>$. \ Then $x$ is a finite sum
of products of the form
\[
\left[
{\displaystyle\sum\limits_{j=1}^{m}}
l_{j}\alpha_{j}\beta_{j}^{\ast}+%
{\displaystyle\sum\limits_{i=1}^{n}}
t_{i}\gamma_{i}u_{i}^{H}\delta_{i}^{\ast}\right]  \cdot\left[
{\displaystyle\sum\limits_{k=1}^{n^{\prime}}}
l_{k}^{\prime}p_{k}q_{k}^{\ast}h_{\lambda}(c_{\lambda})r_{k}s_{k}^{\ast
}\right]  ,
\]
where $l_{k}^{\prime}\in K$ and $p_{k},q_{k},r_{k},s_{k}$ are all paths in
$E$. A typical term in this product is of the form
\[
l_{j}l_{k}^{\prime}\alpha_{j}\beta_{j}^{\ast}p_{k}q_{k}^{\ast}h_{\lambda
}(c_{\lambda})r_{k}s_{k}^{\ast}+t_{i}l_{k}^{\prime}\gamma_{i}u_{i}^{H}%
\delta_{i}^{\ast}p_{k}q_{k}^{\ast}h_{\lambda}(c_{\lambda})r_{k}s_{k}^{\ast}.
\]
Now $s(q_{k})=r(q_{k}^{\ast})\in c_{\lambda}^{0}$ and since $c_{\lambda}$ has
no exits, $r(p_{k})=r(q_{k})\in c_{\lambda}^{0}$.

If $\beta_{j}^{\ast}p_{k}\neq0$, then either $\beta_{j}=p_{k}\beta^{\prime}$
or $p_{k}=\beta_{j}p^{\prime}$. If $\beta_{j}=p_{k}\beta^{\prime}$, then
$\beta^{\prime}$ will be path from a vertex in $c_{\lambda}$ to the vertex
$r(\beta_{j})\in H$ contradicting the fact that $c_{\lambda}$ has no exits.
Similarly, if $p_{k}=\beta_{j}p^{\prime}$, $p^{\prime}$ will be a path from
$r(\beta_{j})\in H$ to $r(p^{\prime})\in c_{\lambda}^{0}$ and this implies, as
$H$ is hereditary, $c_{\lambda}^{0}\subseteq H$, again a contradiction. So
$l_{j}l_{k}^{\prime}\alpha_{j}\beta_{j}^{\ast}p_{k}q_{k}^{\ast}h_{\lambda
}(c_{\lambda})r_{k}s_{k}^{\ast}=0$.

Likewise, if $\delta_{i}^{\ast}p_{k}\neq0$, then either $\delta_{i}%
=p_{k}\delta^{\prime}$ or $p_{k}=\delta_{i}p^{\prime}$. If $\delta_{i}%
=p_{k}\delta^{\prime}$, $\delta^{\prime}$ will be a path from $r(p_{k})\in
c_{\lambda}^{0}$ to the vertex $r(\delta_{i})\in S$, contradicting that
$c_{\lambda}$ has no exits. If $p_{k}=\delta_{i}p^{\prime}$, then $p^{\prime}$
will be a path from $r(\delta_{i})=u_{i}$ to $r(p_{k})\in c_{\lambda}^{0}$.
Observe that in this case, $p^{\prime}=fp^{\prime\prime}$ where $f$ is the
initial edge of $p^{\prime}$ satisfying $r(f)\notin H$. Then%
\[
u_{i}^{H}\delta_{i}^{\ast}p_{k}=u_{i}^{H}p^{\prime}=(v_{i}-%
{\displaystyle\sum\limits_{s(e)=v_{i},r(e)\notin H}}
ee^{\ast})fp^{\prime\prime}=fp^{\prime\prime}-fp^{\prime\prime}=0\text{.}%
\]
So $t_{i}l_{k}^{\prime}\gamma_{i}v_{i}^{H}\delta_{i}^{\ast}p_{k}q_{k}^{\ast
}h_{\lambda}(c_{\lambda})r_{k}s_{k}^{\ast}=0$. Hence $I(H,S)\cdot<h_{\lambda
}(c_{\lambda})>=0$, thus proving our claim.

Let $C$ be the ideal in $L$ such that
\[
C/gr(A)=\bar{C}=%
{\displaystyle\sum\limits_{i\in T}}
<f_{i}(c_{i})>+%
{\displaystyle\sum\limits_{j\in X\backslash T}}
<f_{j}^{\prime}(c_{j})>\subseteq\bar{L}.
\]
Then, using the claim $(\ast)$ in $\bar{L}$, Lemma
\ref{graded multiplicatin ring} and the fact that the index set
$Y=(Y\backslash X)\cup(X\backslash T)$, we have
\begin{align*}
&  \bar{B}\bar{C}\\
&  =\left[  gr(\bar{B})+%
{\displaystyle\sum\limits_{j\in Y\backslash X}}
<g_{j}(c_{j})>+%
{\displaystyle\sum\limits_{j\in X\backslash T}}
<g_{j}(c_{j})>\right]  \times\\
&  \left[
{\displaystyle\sum\limits_{i\in T}}
<f_{i}(c_{i})>+%
{\displaystyle\sum\limits_{j\in X\backslash T}}
<f_{j}^{\prime}(c_{j})>\right] \\
&  =gr(\bar{B})[%
{\displaystyle\sum\limits_{i\in T}}
<f_{i}(c_{i})>]+[%
{\displaystyle\sum\limits_{j\in X\backslash T}}
<g_{j}(c_{j})>][%
{\displaystyle\sum\limits_{i\in T}}
<f_{i}(c_{i})>]\\
+[%
{\displaystyle\sum\limits_{j\in X\backslash T}}
&  <g_{j}(c_{j})>][%
{\displaystyle\sum\limits_{j\in X\backslash T}}
<f_{j}^{\prime}(c_{j})>]\text{, by claim (*) }\\
&  =%
{\displaystyle\sum\limits_{i\in T}}
<f_{i}(c_{i})>+[%
{\displaystyle\sum\limits_{j\in X\backslash T}}
<g_{j}(c_{j})>][%
{\displaystyle\sum\limits_{i\in T}}
<f_{i}(c_{i})>]+%
{\displaystyle\sum\limits_{j\in X\backslash T}}
<f_{j}(c_{j})>\\
&  =%
{\displaystyle\sum\limits_{i\in X}}
<f_{i}(c_{i})>+0=\bar{A}\text{, }%
\end{align*}
as $[%
{\displaystyle\sum\limits_{j\in X\backslash T}}
<g_{j}(c_{j})>][%
{\displaystyle\sum\limits_{i\in T}}
<f_{i}(c_{i})>]\subseteq\lbrack%
{\displaystyle\sum\limits_{j\in X\backslash T}}
<g_{j}(c_{j})>]gr(\bar{B})=0$, by our claim $(\ast)$. It then follows that
$A=BC$.
\end{proof}

From Theorem \ref{LPAs are Multiplication rings}, we obtain the following
interesting property of prime ideals in $L$.

\begin{corollary}
\label{Prime property 1} If $P$ is a prime ideal of $L$, then for any ideal
$A$ with $P\subsetneqq A$, we have $P=AP$.
\end{corollary}

\begin{proof}
Since $L$ is a multiplication ring, there is an ideal $C$ of $L$ such that
$P=AC$. Since $P$ is a prime ideal and $A\nsubseteq P$, we conclude that
$C\subseteq P$. Then
\[
P=AC\subseteq AP\subseteq P\text{.}
\]
Hence $P=AP$.
\end{proof}

\section{Prime, Irreducible and Primary Ideals of a Leavitt Path Algebra}

In this section, we investigate special types of ideals in $L$ such as the
prime, the irreducible and the primary ideals. While these three concepts are
independent for ideals in a commutative ring, we show that they coincide for
graded ideals in the Leavitt path algebra $L$. We also show that a non-graded
ideal $I$ of $L$ is irreducible if and only if $I$ is a primary ideal if and
only if $I=P^{n}$, a power of a prime ideal $P$. This is useful in the
factorization of ideals in the next section. We also point out some
interesting properties of the prime ideals in $L$.

The following description of prime ideals of $L$ was given in \cite{R1}.

\begin{theorem}
\label{R-1}(Theorem 3.2, \cite{R1}) \ An ideal $P$ of $L=L_{K}(E)$ with $P\cap
E^{0}=H$ is a prime ideal if and only if $P$ satisfies one of the following properties:

(i) $\ \ P=I(H,B_{H})$ and $E^{0}\backslash H$ is downward directed;\ 

(ii) $\ P=I(H,B_{H}\backslash\{u\})$, $v\geq u$ for all $v\in E^{0}\backslash
H$ and the vertex $u^{\prime}$ that corresponds to $u$ in $E\backslash
(H,B_{H}\backslash\{u\})$ is a sink;

(iii) $P$ is a non-graded ideal of the form $\ P=I(H,B_{H})+<p(c)>$, where $c$
is a cycle without exits based at a vertex $u$ in $E\backslash(H,B_{H})$,
$v\geq u$ for all $v\in E^{0}\backslash H$ and $p(x)$ is an irreducible
polynomial in $K[x,x^{-1}]$ such that $p(c)\in P$.
\end{theorem}

We shall use the observation that if $E^{0}\backslash H$ contains a cycle
without exits, then in Theorem \ref{R-1} the case (ii) will not occur.

Next we point out some interesting properties related to prime ideals of $L$.

\begin{lemma}
\label{P^n < A < P => A=P^r} Suppose $P$ is a prime ideal of $L$ and $A$ is an
ideal such that, for some integer $n>1$, $P^{n}\subseteq A\subseteq P$. Then
$A=P^{r}$ for some $1\leq r\leq n$.
\end{lemma}

\begin{proof}
If $P$ is graded, then $P=P^{n}=A$. So assume that $P$ is a non-graded prime
ideal. By Theorem \ref{R-1}, $P=gr(P)+<p(c)>$ with $gr(P)=I(H,B_{H})$, where
$c$ is a cycle without exits in $E\backslash(H,B_{H})$ and $p(x)$ is an
irreducible polynomial in $K[x]$. By Lemma \ref{Product = Intersection},
$gr(P^{n})=gr(P)$. As
\[
P^{n}/gr(P)=(P/gr(P))^{n}=<p(c)>^{n}%
\]
in $L/gr(P)$, $P^{n}=gr(P)+<p(c)>^{n}$. Now, $A$ must be a non-graded ideal.
Because if $A$ is graded, then $A\subseteq gr(P)\subseteq P^{n}$ and this
implies $A=P^{n}$, a contradiction as $P^{n}$ is not graded. Hence
$A=gr(P)+<f(c)>$ where $f(x)\in K[x]$. In $L/gr(P)$, we have $<p(c)>^{n}%
\subseteq<f(c)>\subseteq<p(c)>$ and hence $f(x)$ is a divisor of $p^{n}(x)$ in
$K[x]$. So $f(x)=p^{r}(x)$ for some $r\leq n$ and $A=gr(P)+<p(c)>^{r}=P^{r}$.
\end{proof}

\begin{proposition}
\label{Prime Property 2} Let $P$ be a prime ideal of a Leavitt path algebra
$L$. Then for any ideal $A$ with $P\subseteq A$, either $P=A$ or $P\subseteq
gr(A)$. \ 
\end{proposition}

\begin{proof}
If $P$ is graded, then clearly $P$ possess the desired properties. So assume
that $P$ is a non-graded ideal. Suppose there is an ideal $A$ such that
$P\subseteq A$ and $P\nsubseteq gr(A)$. Clearly $A$ is non-graded and so we
have, by Theorem \ref{R-2}, $A=I(H^{\prime},S^{\prime})+%
{\displaystyle\sum\limits_{i\in X}}
<f_{i}(c_{i})>$ where $H^{\prime}=A\cap E^{0}$, $X$ is some index set, for
each $i$, $f_{i}(x)\in K[x]$ and $c_{i}$ is a cycle without exits in
$E^{0}\backslash H^{\prime}$. By Theorem \ref{R-1} (iii) , $P=I(H,B_{H}%
)+<p(c)>$, where $H=P\cap E^{0}$, $c$ is a cycle without exits based at a
vertex $u$ in $E\backslash(H,B_{H})$, $v\geq u$ for all $v\in E^{0}\backslash
H$ and $p(x)$ is an irreducible polynomial in $K[x]$. Clearly, $c$ is the only
cycle without exits in $E^{0}\backslash H$ based at a vertex $u$. \ As
$P\nsubseteqq gr(A)=I(H^{\prime},S^{\prime})$, $p(c)\notin I(H^{\prime
},S^{\prime})$ and $c$ is a cycle without exits in $E^{0}\backslash H^{\prime
}$ also. Since $v\geq u$ for all $v\in E^{0}\backslash H^{\prime}$, $c_{i}=c$
for all $i$. Moreover, $H^{\prime}=H$ because if there is a $v\in H^{\prime
}\backslash H$, then $v\geq u$ implies that $u\in H^{\prime}$, a
contradiction. In this case $S^{\prime}=B_{H}$. Thus $A$ is of the form
$A=I(H,B_{H})+<f(c)>$ where $f(x)\in K[x]$. Clearly, in $L/I(H,B_{H})$,
$<p(c)>\subseteq<f(c)>$. Since, by Theorem \ref{R-2}, $f(x)$ is a polynomial
of smallest degree in $K[x]$ such that $f(c)\in A$ and since $p(x)$ is
irreducible in $K[x]$, $<p(c)>=<f(c)>$ in $L/I(H,B_{H})$. Then $A=I(H,B_{H}%
)+<p(c)>=P$, as desired.
\end{proof}

\textbf{Remark}: Corollary \ref{Prime property 1} (that, for a prime ideal
$P$, $P=PA$ for any ideal $A\supsetneqq P$) can also be derived from
Proposition \ref{Prime Property 2}, as $P=P\cap gr(A)=Pgr(A)\subseteq
PA\subseteq P$.

Recall, an ideal $I$ of a ring $R$ is called an \textbf{irreducible ideal} if,
for ideals $A,B$ of $R$, $I=A\cap B$ implies that either $I=A$ or $I=B$. Given
an ideal $I$, the \textbf{radical }$Rad(I)$ of $I$ is the intersection of all
prime ideals containing $I$. A useful property is that if $a\in Rad(I) $, then
$a^{n}\in I$ for some integer $n\geq0$ (see \cite{Lam}). An ideal $I $ of $R$
is said to be a \textbf{primary ideal} if, for any two ideals $A,B$, if
$AB\subseteq I$ and $A\nsubseteqq I$, then $B\subseteq Rad(I)$.

\begin{lemma}
\label{gr(Rad(I))=gr(I)} For any ideal $I$ of $L$, $gr(Rad(I))=gr(I)$.
\end{lemma}

\begin{proof}
Clearly $gr(I)\subseteq gr(Rad(I))$. On the other hand, if $a\in Rad(I)$, then
$a^{n}\in I$ for some integer $n\geq0$. This means that every idempotent
element in $Rad(I)$ belongs to $I$. Since the graded ideal $gr(Rad(I))$ is
generated by idempotents, $gr(Rad(I))\subseteq I$ and hence $gr(Rad(I))=gr(I)$.
\end{proof}

\textbf{Remark}: We note in passing that for any graded ideal $I$ of $L$, say
$I=I(H,S)$, $Rad(I)=I$. Because, $Rad(I)/I$ is a nil ideal in $L/I$ and
$L/I\cong L_{K}(E\backslash(H,S))$ has no non-zero nil ideals.

\begin{lemma}
\label{gr(Primary) is prime} Let $I$ be a primary or an irreducible ideal of
$L$. Then $gr(I)$ is a prime ideal.
\end{lemma}

\begin{proof}
Suppose $I$ is a primary ideal. In view of Proposition II.1.4, Chapter II in
\cite{NO}, we need only to show that $gr(I)$ is graded prime. Consider two
graded ideals $A,B$ such that $AB\subseteq gr(I)$ and $A\nsubseteqq gr(I)$. As
$A$ is graded, $A\nsubseteqq I$. Since $I$ is primary, $B\subseteq Rad(I)$. As
$B$ is a graded ideal, we have $B\subseteq gr(Rad(I))=gr(I)$, by Lemma
\ref{gr(Rad(I))=gr(I)}. Hence $gr(I)$ is a prime ideal.

Suppose now that $I$ is an irreducible ideal. As before, we need only to show
that $gr(I)$ is graded prime. Suppose $A,B$ are graded ideals of $L$ such that
$AB\subseteq gr(I)$. If both $A\nsubseteqq gr(I)$ and $B\nsubseteqq gr(I)$,
then again $A\nsubseteqq I$ and $B\nsubseteqq I$. By distributive law (Theorem
\ref{LPA is Arithmetic}) we then have,
\begin{align*}
(I+A)\cap(I+B)  &  =(I\cap I)+(I\cap B)+(A\cap I)+(A\cap B)\\
&  =I+(AB)\text{, as }(A\cap B)=AB\text{, by Lemma
\ref{Product = Intersection}(i)}\\
&  =I\text{.}%
\end{align*}
This contradicts the fact that $I$ is irreducible. Hence $gr(I)$ is a prime ideal.
\end{proof}

\begin{corollary}
\label{For graded I, primary=irred=prime} Suppose $I$ is a graded ideal of $L
$. Then the following are equivalent:

(a) $I$ is a primary ideal;

(b) $I$ is a prime ideal;

(c) $I$ is an irreducible ideal.
\end{corollary}

We are now ready to prove the main result of this section which extends the
above corollary to arbitrary ideals of $L$.

\begin{theorem}
\label{Irred = primary =prime power} Let $L=L_{K}(E)$ be the Leavitt path
algebra of an arbitrary graph $E$. Then the following properties are
equivalent for an ideal $I$ of $L$:

(i) $\ \ \ I$ is an irreducible ideal;

(ii) \ $\ I=P^{n}$, a power of a prime ideal $P$ for some $n\geq1$;

(iii) \ $I$ is a primary ideal.
\end{theorem}

\begin{proof}
(i) =
$>$
(ii). Suppose $I$ is an irreducible ideal. If $I$ is a graded ideal, then $I$
must be a prime ideal, by Corollary \ref{For graded I, primary=irred=prime}.

Suppose $I$ is a non-graded irreducible ideal. By Theorem \ref{R-2},
$I=I(H,S)+\sum_{i\in X}<f_{i}(c_{i})>$, where for each $i$, $c_{i}$ is a cycle
without exits based at a vertex $v_{i}$ in $E\backslash(H,S)$ and $f_{i}(x)\in
K[x]$. From Lemma \ref{gr(Primary) is prime}, $gr(I)$ is a graded prime ideal.
By Theorem \ref{R-1}, we then have $gr(I)=I(H,B_{H})$ and $E\backslash
(H,B_{H})$ is downward directed. Hence there can be only one cycle, say $c$
based at a vertex $v$ and without exits in $E\backslash(H,B_{H})$. Thus $I$ is
of the form $I=I(H,B_{H})+<f(c)>$ for some polynomial $f(x)\in K[x]$. Let
$f(x)=p_{1}^{k_{1}}(x)\cdot\cdot\cdot p_{m}^{k_{m}}(x)$ be a factorization of
$f(x)$ as a product of powers of distinct irreducible polynomials $p_{i}(x)$
in $K[x]$. We claim $m=1$, that is, $f(x)$ is a power of a single irreducible
polynomial. Assume, on the contrary, $m>1$. Let $g(x)=$ $p_{1}^{k_{1}}(x)$ and
$h(x)=p_{2}^{k_{2}}(x)\cdot\cdot\cdot p_{m}^{k_{m}}(x)$. Note that each
$p_{i}(x)$ is still irreducible in $K[x,x^{-1}]$. Clearly, $<f(x)>=(g(x)>\cap
<h(x)>$ in $K[x,x^{-1}]$. If $M$ is the (graded) ideal generated by $c^{0}$ in
$L/I(H,B_{H})$, then $M$ contains the ideals $<f(c)>,<g(c)>,<h(c)>$ and, by
Theorem \ref{Ideal generated by c^0- c cycle withour exts}, $M\cong
M_{\Lambda}(K[x,x^{-1}])$ which is Morita equivalent to $K[x,x^{-1}]$. As the
ideal lattices of $M$ and $K[x,x^{-1}]$ are isomorphic, we then conclude that,
in $M$ and hence in $L/I(H,B_{H})$, $<f(c)>=<g(c)>\cap<h(c)>$. Let
$A=I(H,B_{H})+<g(c)>$ and $B=I(H,B_{H})+<h(c)>$. Then $[A/I(H,B_{H}%
)]\cap\lbrack B/I(H,B_{H})]=I/I(H,B_{H})=<f(c)>$ and so $A\cap B=I$. Since
$A\neq I$ and $B\neq I$, this contradicts that $I$ is irreducible. Hence
$I=I(H,B_{H})+<p^{n}(c)>$ where $p(x)$ is an irreducible polynomial and $c$ is
a cycle without exits in $E^{0}\backslash H$ and $E^{0}\backslash H$ is
downward directed. It is then clear that $P=I(H,B_{H})+<p(c)>$ is a prime
ideal and that $gr(P^{n})=[I(H,B_{H})]^{n}=I(H,S)$. In $L/I(H,S)$, using Lemma
\ref{Pricipal ideals commute}, we have
\[
P^{n}/I(H,S)=(P/I(H,S))^{n}=<p(c)>^{n}=<p^{n}(c)>.
\]
Consequently, $P^{n}=I(H,S)+<p^{n}(c)>=I$. This proves (ii)

(ii) =
$>$
(iii) Suppose $I=P^{n}$ where $P$ is a prime ideal. Consider two ideals $A,B$
such that $AB\subseteq I\subseteq P$. If $A\nsubseteqq P$, then $B\subseteq
P$. But $P=Rad(I)$. Hence $I$ is a primary ideal.

(iii) =
$>$
(ii). Suppose $P$ is a primary ideal. By Lemma \ref{gr(Primary) is prime} and
Corollary \ref{For graded I, primary=irred=prime}, we may assume that $I$ is a
non-graded ideal such that $gr(I)=I(H,B_{H})$ is a prime ideal. Consequently,
$E\backslash(H,B_{H})$ is downward directed and hence can have no more than
one cycle, say $c$ without exits in $E^{0}\backslash(H,B_{H})$. Then $I$ will
be of the form $I=gr(I)+<f(c)>$ where $f(x)\in K[x]$. We claim that $f(x)$ is
a power of an irreducible polynomial. Suppose, on the contrary, $f(x)=p_{1}%
^{k_{1}}(x)\cdot\cdot\cdot p_{n}^{k_{n}}(x)$ is a factorization of $f(x)$ as a
product of powers of distinct irreducible polynomials $p_{1}(x),\cdot
\cdot\cdot,p_{n}(x)$ where $n>1$. Let $g(x)=p_{1}^{k_{1}}(x)$ and
$h(x)=p_{2}^{k_{2}}(x)\cdot\cdot\cdot p_{n}^{k_{n}}(x)$. Let $A=gr(I)+<g(c)>$
and $B=gr(I)+<h(c)>$. Using Lemma \ref{Pricipal ideals commute}, we have in
$L/gr(I)$,
\[
(A/gr(I))(B/gr(I))=<g(c)><h(c)>=<f(c)>=I/gr(I).
\]
Consequently, $AB=I$. But neither $A/gr(I)$ nor $B/gr(I)$ is contained in
$\ Rad(I)/gr(I)$ which is the ideal\ generated by the product $p_{1}%
(c)\cdot\cdot\cdot\cdot p_{n}(c)$. This contradicts that $I$ is a primary
ideal. Hence, $I=gr(I)+<p^{n}(c)>$ where $p(x)$ is an irreducible polynomial
in $K[x]$ and $n\geq0$. Then $I=P^{n}$ where $P=gr(I)+<p(c)>$ is a prime
ideal. This proves (ii).

(ii) =
$>$
(i) Suppose $I=P^{n}$ where $P$ is a prime ideal. Let $I=A\cap B$ for some
ideals $A,B$ in $L$. Since $I=P\cap I=(P\cap A)\cap(P\cap B)$, we may assume,
without loss of generality, that both $A,B\subseteq P$. Thus $P^{n}\subseteq
A,B\subseteq P$ and so, by Lemma \ref{P^n < A < P => A=P^r}, $A=P^{r}$ and
$B=P^{s}$ for some $r,s\leq n.$Then $P^{r}\cap P^{s}=P^{n}$ implies one of $r$
or $s$ must be $n$. Thus $I=A$ or $B$. Hence $I$ is irreducible.
\end{proof}

\section{Factorization of Ideals in L}

As noted in the Introduction, ideals in an arithmetical ring admit interesting
representations as products of special types of ideals. In this section, we
explore the factorization of ideals in a Leavitt path algebra $L$ as products
of prime ideals and as products of irreducible/primary ideals. The prime
factorization of graded ideals of $L$ seems to influence that of the
non-graded ideals in $L$. Indeed, an ideal $I$ is a product of prime ideals in
$L$ if and only its graded part $gr(I)$ has the same property and, moreover,
$I/gr(I)$ is finitely generated with a generating set of cardinality no more
than the number of distinct prime ideals in an irredundant factorization of
$gr(I)$. We also show that $I$ is an intersection of irreducible ideals if and
only if $I$ is an intersection of prime ideals. If $L$ is the Leavitt path
algebra of a finite graph or, more generally, if $L$ is two-sided noetherian
or two-sided artinian, then every ideal of $L$ is shown to be a product of
prime ideals. The uniqueness of such factorizations was discussed in
\cite{EER}.

We begin with the following useful proposition.

\begin{proposition}
\label{gr(I) prime implies I product of primes}Suppose $I$ is a non-graded
ideal of $L$. If $gr(I)$ is a prime ideal, then $I$ is a product of prime ideals.
\end{proposition}

\begin{proof}
By Theorem \ref{R-2}, $I=I(H,S)+%
{\displaystyle\sum\limits_{t\in T}}
<f_{t}(c_{t})>$, where $T$ is some index set, for each $t\in T,c_{t}$ is a
cycle without exits in $E\backslash(H,S)$, $c_{t}^{0}\cap c_{s}^{0}=\emptyset$
for $t\neq s$ and $f_{t}(x)\in K[x]$. Now $gr(I)=I(H,S)$. If $I(H,S)$ is a
prime ideal, then $E^{0}\backslash H$ is downward directed (Theorem \ref{R-1})
and so there can be only one cycle $c$ without exits in $E\backslash(H,S)$
based at some vertex $v$. This means that $I(H,S)=I(H,B_{H})$ (Theorem
\ref{R-1}) and $I$ must be of the form $I=I(H,B_{H})+<f(c)>$ where $f(x)\in
K[x]$. Let $f(x)=p_{1}(x)\cdot\cdot\cdot p_{n}(x)$ be a factorization of
$f(x)$ as a product of irreducible polynomials $p_{i}(x)$ in $K[x]$. Note that
each $p_{i}(x)$ is irreducible in $K[x,x^{-1}]$. Now, for each $j$,
$P_{j}=I(H,B_{H})+<p_{j}(c)>$ is a prime ideal (Theorem \ref{R-1}). Clearly $%
{\displaystyle\prod\limits_{j=1}^{n}}
P_{j}\supseteq I$. Using Lemma \ref{Pricipal ideals commute} and a simple
induction on $n$, we have, in $L/I(H,B_{H})$,
\begin{align*}%
{\displaystyle\prod\limits_{j=1}^{n}}
P_{j}/I(H,B_{H})  &  =<p_{1}(c)>\cdot\cdot\cdot<p_{n}(c)>\\
&  =<p_{1}(c)\cdot\cdot\cdot p_{n}(c)>=<f(c)>=I/I(H,B_{H}).
\end{align*}
Hence $I=%
{\displaystyle\prod\limits_{j=1}^{n}}
P_{j}$ is a product of prime ideals.
\end{proof}

\begin{theorem}
\label{gr(I) product of primes => I product of primes} Let $E$ be an arbitrary
graph. For a non-graded ideal $I$ of $L:=L_{k}(E)$, the following are equivalent:

(a) $I$ is a product of prime ideals;

(b) $I$ is a product of primary ideals;

(c) $I$ is a product of irreducible ideals;

(d) $gr(I)$ is a product of (graded) prime ideals;

(e) $gr(I)=P_{1}\cap\cdot\cdot\cdot\cap P_{m}$ is an irredundant intersection
of $m$ graded prime ideals $P_{j}$ and $I/gr(I)$ is generated by at most $m$
elements and is of the form $I/gr(I)=%
{\displaystyle\bigoplus\limits_{r=1}^{k}}
<f_{r}(c_{r})>$ where $k\leq m$ and, for each $r=1\cdot\cdot\cdot k$, $c_{r}$
is a cycle without exits in $E^{0}\backslash I$ and $f_{r}(x)\in K[x]$ is a
polynomial of smallest degree such that $f_{r}(c_{r})\in I$.
\end{theorem}

\begin{proof}
Now (a) =%
$>$
(b) =%
$>$
(c), since every prime ideal is primary and primary ideals in $L$ are
irreducible, by Theorem \ref{Irred = primary =prime power}.

(c) =%
$>$
(d). Suppose $I=P_{1}\cdot\cdot\cdot P_{n}$ is \ a product of irreducible
ideals. Clearly $gr(I)\subseteq gr(P_{j})$ for all $j=1,\cdot\cdot\cdot,n$ and
so $gr(I)\subseteq%
{\displaystyle\bigcap\limits_{j=1}^{n}}
gr(P_{j})$. On the other hand, by Lemma \ref{Product = Intersection}, $%
{\displaystyle\bigcap\limits_{j=1}^{n}}
gr(P_{j})=%
{\displaystyle\prod\limits_{j=1}^{n}}
gr(P_{j})\subseteq I$ and is a graded ideal. So $%
{\displaystyle\bigcap\limits_{j=1}^{n}}
gr(P_{j})\subseteq gr(I)$, by Theorem \ref{R-2}. Thus $gr(I)=%
{\displaystyle\bigcap\limits_{j=1}^{n}}
gr(P_{j})=%
{\displaystyle\prod\limits_{j=1}^{n}}
gr(P_{j})$. Now, by Lemma \ref{gr(Primary) is prime}, each $gr(P_{j})$ is a
prime ideal. Thus $gr(I)$ is a product/intersection of (graded) prime ideals.

(d) =%
$>$
(e). Suppose $gr(I)=P_{1}\cdot\cdot\cdot P_{n}$ is a product of graded prime
ideals of $L$. By Lemma \ref{Product = Intersection}, $gr(I)=P_{1}\cap
\cdot\cdot\cdot\cap P_{n}$. If needed, remove appropriate ideals $P_{j}$ and
assume that $gr(I)=P_{1}\cap\cdot\cdot\cdot\cap P_{m}$ is an irredundant
intersection of graded prime ideals $P_{j}$ which are thus all distinct and
none contains the other ideals. Let $H=E^{0}\cap I$ and $S=\{v\in B_{H}%
:v^{H}\in I\}$ so that $gr(I)=I(H,S)$. By Theorem \ref{R-2}, the non-graded
ideal $I$ is of the form $I=I(H,S)+%
{\displaystyle\sum\limits_{t\in T}}
<f_{t}(c_{t})>$, where $T$ is some index set and, for each $t$, $c_{t}$ is a
cycle without exits in $E\backslash(H,S)$ based at a vertex $v_{t}$ and
$c_{t}^{0}\cap c_{s}^{0}=\emptyset$ if $t\neq s$. Now, for each $t\in T$,
there must exist an index $j_{t}$ depending on $t$, such that $c_{t}\notin
P_{j_{t}}$. Because, otherwise, $c_{t}\in%
{\displaystyle\bigcap\limits_{i=1}^{m}}
P_{i}=gr(I)=I(H,S)$, a contradiction. Let $P_{j_{t}}\cap E^{0}=H_{j_{t}}$.
Then $E^{0}\backslash H_{j_{t}}$ is downward directed, as $P_{j_{t}}$ is a
prime ideal. Now $v_{t}\in E^{0}\backslash H_{j_{t}}$ and, since $c_{t}$ is a
cycle without exits in $E^{0}\backslash H_{j_{t}}$, we have $u\geq v_{t}$ for
all $u\in E^{0}\backslash H_{j_{t}}$. From the description of the prime ideals
in Theorem \ref{R-1}, we then conclude that $P_{j_{t}}=I(H_{j_{t}}%
,B_{H_{j_{t}}})$. We claim that $v_{t}\in P_{j}$ for all $j\neq j_{t}$.
Suppose, on the contrary, $v_{t}\notin P_{j}$ for some $j\neq j_{t}$. Let
$H_{j}=P_{j}\cap E^{0}$. Since $E^{0}\backslash H_{j}$ is downward directed,
we have $u\geq v_{t}$ for every $u\in E^{0}\backslash H_{j}$ and
$P_{j}=I(H_{j},B_{H_{j}})$. If $P^{\prime}=P_{j_{t}}\cap P_{j}$ with
$H^{\prime}=P^{\prime}\cap E^{0}$, then every $u\in E^{0}\backslash H^{\prime
}=E^{0}\backslash(H_{j_{t}}\cap H_{j})=(E^{0}\backslash H_{j_{t}})\cup
(E^{0}\backslash H_{j})$ satisfies $u\geq v_{t}$. This means that
$E^{0}\backslash H^{\prime}$ is downward directed and $B_{H^{\prime}%
}=B_{H_{j_{t}}}\cap B_{H_{j}}$. Hence $P^{\prime}=I(H^{\prime},B_{H^{\prime}%
})$ is a prime ideal. But then $P_{j_{t}}\cdot P_{j}=P_{j_{t}}\cap
P_{j}=P^{\prime}$ implies $P_{j_{t}}\subseteq P^{\prime}$ or $P_{j}\subseteq
P^{\prime}$. This means either $P_{j_{t}}\subseteq P_{j}$ or $P_{j}\subseteq
P_{j_{t}}$, contradicting the fact that $P_{1}\cap\cdot\cdot\cdot\cap P_{m}$
is an irredundant intersection. Hence, for each $t\in T$, $v_{t}\notin
P_{j_{t}}$ but $v_{t}\in P_{j}$ for all $j\neq j_{t}$. Also, if $s\in T$ with
$s\neq t$ (so $c_{s}\neq c_{t}$), then $P_{j_{s}}\neq P_{j_{t}}$. Thus
$|T|\leq m$, the number of prime ideals $P_{j}$. Hence $T$ must be a finite
set. Thus $I/gr(I)$ is generated by the finite set $\{f_{j}(c_{j}):j\in
T\subseteq\{1,\cdot\cdot\cdot,m\}\}$. Now each $<f_{j}(c_{j})>$ is an ideal in
the ideal $A_{j}$ generated by the vertices on the cycle $c_{j}$. It was shown
in the proof of Theorem \ref{Ideal generated by c^0- c cycle withour exts}
that $%
{\displaystyle\sum\limits_{{}}}
A_{j}=%
{\displaystyle\bigoplus\limits_{{}}}
A_{j}$. Hence $I/gr(I)=%
{\displaystyle\bigoplus\limits_{r=1}^{k}}
<f_{r}(c_{r})>$ where $k\leq m$. This proves (e).

(e) =%
$>$
(a). Suppose $gr(I)=P_{1}\cap\cdot\cdot\cdot\cap P_{m}$ is an irredundant
intersection of graded prime ideals $P_{j}$ and $I/gr(I)=%
{\displaystyle\bigoplus\limits_{r=1}^{k}}
<f_{r}(c_{r})>$ where $k\leq m$ and, for each $r=1\cdot\cdot\cdot k$, $c_{r}$
is a cycle without exits based at a vertex $v_{r}$ in $E^{0}\backslash(I\cap
E^{0})$ and $f_{r}(x)\in K[x]$. Thus $I=(P_{1}\cap\cdot\cdot\cdot\cap P_{m})+%
{\displaystyle\sum\limits_{r=1}^{k}}
<f_{r}(c_{r})>$. From the proof of (d) =
$>$
(c), we can assume, after re-indexing, that for each $r$, $v_{r}\notin P_{r}$
and that $v_{r}\in P_{s}$ for all $s\neq r$. For each $r=1,\cdot\cdot\cdot,k$,
define $Q_{r}=P_{r}+<f_{r}(c_{r})>$. By Proposition
\ref{gr(I) prime implies I product of primes}, each ideal $Q_{r}$ is a product
of prime ideals. So we are done if we show that%
\[
(P_{1}\cap\cdot\cdot\cdot\cap P_{m})+%
{\displaystyle\sum\limits_{r=1}^{k}}
<f_{r}(c_{r})>=Q_{1}\cdot\cdot\cdot Q_{k}P_{k+1}\cdot\cdot\cdot P_{m}.
\]
We prove this by induction on $k$. Suppose $k=1$. Consider $Q_{1}P_{2}%
\cdot\cdot\cdot P_{m}$. Now $A=P_{2}\cdot\cdot\cdot P_{m}=P_{2}\cap\cdot
\cdot\cdot\cap P_{m}$ is a graded ideal and $v_{1}\in A$. Using Lemma
\ref{Product = Intersection} and the fact that $<f_{1}(c_{1})>\subseteq A$, we
get
\begin{align*}
Q_{1}P_{2}\cdot\cdot\cdot P_{m}  &  =Q_{1}A=Q_{1}\cap A=(P_{1}+<f_{1}%
(c_{1})>)\cap A\\
&  =(P_{1}\cap A)+<f_{1}(c_{1})>\\
&  =(P_{1}\cap\cdot\cdot\cdot\cap P_{m})+<f_{1}(c_{1})>.
\end{align*}
Suppose $k>1$ and assume that the statement holds for $k-1$ so that
\[
Q_{1}\cdot\cdot\cdot Q_{k-1}P_{k}\cdot\cdot\cdot P_{m}=(P_{1}\cap\cdot
\cdot\cdot\cap P_{m})+%
{\displaystyle\sum\limits_{r=1}^{k-1}}
<f_{r}(c_{r})\text{.}%
\]
Then
\begin{align*}
&  Q_{1}\cdot\cdot\cdot Q_{k}P_{k+1}\cdot\cdot\cdot P_{m}\\
&  =Q_{1}\cdot\cdot\cdot Q_{k-1}(P_{k}+<f_{k}(c_{k})>)P_{k+1}\cdot\cdot\cdot
P_{m}\\
&  =Q_{1}\cdot\cdot\cdot Q_{k-1}P_{k}\cdot\cdot\cdot P_{m}+Q_{1}\cdot
\cdot\cdot Q_{k-1}(<f_{k}(c_{k})>)P_{k+1}\cdot\cdot\cdot P_{m}\\
&  =(P_{1}\cap\cdot\cdot\cdot\cap P_{m})+%
{\displaystyle\sum\limits_{r=1}^{k-1}}
<f_{r}(c_{r})>+(<f_{k}(c_{k})>)\text{,}%
\end{align*}
due to the fact that $<f_{k}(c_{k})>\subseteq P_{j}$ for all $j\neq k$ and
that
\[
<f_{k}(c_{k})>P_{j}=P_{j}<f_{k}(c_{k})>=<f_{k}(c_{k})>,
\]
by Lemma \ref{Product = Intersection}(i). This shows that $I$ is a product of
prime ideals, thus proving (a).
\end{proof}

\textbf{Remark}: It is clear from the above theorem that if a graded ideal
$I(H,S)$ is a product of prime ideals, then there will necessarily be at most
finitely many cycles without exits in $E\backslash(H,S)$.

As an application of the above theorem, we obtain the following propositions.

\begin{proposition}
\label{Graph Finite}Let $E$ be a finite graph, or more generally, let $E^{0}$
be finite. Then every ideal of $L=L_{K}(E)$ is a product of prime ideals.
\end{proposition}

\begin{proof}
In view of Theorem \ref{gr(I) product of primes => I product of primes}, we
need only to show that every graded ideal $I$ of $L$ is an intersection of
finitely many prime ideals of $L$. Let $I\cap E^{0}=H$. Since $L/I\cong
L_{K}(E\backslash H)$ is a Leavitt path algebra, its prime radical is $0$, and
so $I=\cap\{P:P$ prime ideal $\supseteq I\}$ = $\cap\{P:P$ graded prime ideal
$\supseteq I\}$. Since $(E\backslash H)^{0}$ is finite, there are only
finitely many hereditary saturated subsets of $(E\backslash H)^{0}$ and so
there are only finitely many graded ideals in $L/I$. This means that $I$ is an
intersection of finitely many graded prime ideals.
\end{proof}

\begin{proposition}
\label{L Artinian} Suppose $L$ is (a) two-sided artinian or (b) two-sided
noetherian. Then every ideal of $L$ is a product of prime ideals.
\end{proposition}

\begin{proof}
Assume (a). We first show that $\{0\}$ is an intersection of finitely many
prime ideals. Suppose, on the contrary, no intersection of finitely many prime
ideals of $L$ is $0$. In particular, $\{0\}$ is not a prime ideal. Let
$\boldsymbol{F}=\{Q:Q$ is a non-zero intersection of finitely many prime
ideals$\}$. Note that $\mathbf{F}$ is non-empty since, as observed in Remark
2.9 in \cite{EER} (also from the proof of Proposition 6.5 below), every
Leavitt path algebra always contains a prime ideal. Let $M$ be a minimal
element of $\mathbf{F}$, say $M=P_{1}\cap\cdot\cdot\cdot\cap P_{m}\neq\{0\}$.
Now, for any prime ideal $P$, $M\cap P\neq\{0\}$, by our supposition. So by
the minimality of $M$, $M\cap P=M$ for every prime ideal $P$. This means $M=%
{\displaystyle\bigcap\limits_{P\text{ prime ideal of }L}}
P$ which is $\{0\}$ as the prime radical of $L$ is zero. This contradiction
shows that $\{0\}$ is an intersection of finitely many prime ideals. Now for
every graded ideal $I=I(H,S)$, $L/I\cong L_{K}(E\backslash(H,S))$ is a Leavitt
path algebra which is also two-sided artinian and, by the above argument,
there are finitely many prime ideals in $L/I$ whose intersection is zero. This
means that every graded ideal $I$ is an intersection (and, by Lemma
\ref{Product = Intersection}, a product) of finitely many prime ideals. By
Theorem \ref{gr(I) product of primes => I product of primes}, every ideal in
$L$ is then a product of prime ideals.

Assume (b), so the ideals of $L$ satisfy the ascending chain condition. In
view of Theorem \ref{gr(I) product of primes => I product of primes}, we need
to only to show that every graded ideal $I$ in $L$ is a product of prime
ideals. Since any graded homomorphic image of $L$ is also a two-sided
noetherian Leavitt path algebra, it is enough to show that $\{0\}$ is a
product of finitely many prime ideals in $L$. If $\{0\}$ is a prime ideal, we
are done. Otherwise, we wish to show that $\{0\}$ is a product of finitely
many minimal prime ideals in $L$. We shall use the usual argument given in
such situations in the study of commutative rings. Assume the contrary.
Consider the set%
\[
\mathbf{X}=\{A:A\neq0,A\text{ a product of finitely many minimal prime ideals
in }L\}.
\]
Let
\[
\mathbf{C=\{}J:J\text{ is an ideal of }L\text{ such that }A\nsubseteqq J\text{
for all }A\in\mathbf{X}\}.
\]
Now $\mathbf{C\neq\emptyset}$, as $\{0\}\in\mathbf{C}$. Since the ascending
chain condition holds, we appeal to Zorn's Lemma to get a maximal element $M$
of $\mathbf{C}$. We claim that $M$ is a prime ideal. To see this, suppose
$a\notin M$ and $b\notin M$ so that $M+LaL\supseteq A$ and $M+LbL\supseteq B$
for some $A,B\in\mathbf{X}$. Then $AB\subseteq(M+LaL)(M+LbL)\subseteq
M+LaLbL$. Now, by our assumption, $AB\in\mathbf{X}$ and so $AB\nsubseteqq M$.
Consequently, $aLb\nsubseteqq M$. This shows that $M$ is a prime ideal of $L$.
If $P$ is a minimal prime ideal inside $M$, then $P\neq\{0\}$ (as $\{0\}$ is
not a prime ideal) and so $P$ satisfies $P\in\mathbf{X}$. Since we also have
$P\subseteq M$, we reach a contradiction. $\ $Thus $\{0\}$ is a product of
finitely many (minimal) prime ideals.
\end{proof}

\textbf{Note}: (i) In \cite{C}, it was shown that the Leavitt path algebra of
a finite graph is two-sided noetherian. Using this, one can also derive
Proposition \ref{Graph Finite} from Proposition \ref{L Artinian}.

(ii) The reader might wonder if the descending chain condition on ideals of
$L$ implies the ascending chain condition in which case Proposition 6.4 (b)
will follow from Proposition 6.4 (a). However, these two concepts are
independent for Leavitt path algebras. Indeed, for the graph $F$ in Example
7.2 below, $L_{K}(F)$ is two-sided artinian, but is not two-sided noetherian.
Likewise, if $E$ is the graph with a single vertex $v$ and a loop $c$ based at
$v$, then $L_{K}(E)\cong K[x,x^{-1}]$ is a (commutative) noetherian ring, but
is not an artinian ring.

In the case when $E$ is an arbitrary graph, we have the following description
of the Leavitt path algebra $L=L_{K}(E)$ in which every ideal is a product of
prime ideals.

\begin{proposition}
\label{Prime factorization in LPA} Let $E$ be an arbitrary graph and let
$L=L_{K}(E)$. Then every proper ideal of $L$ is a product of prime ideals if
and only if, every homomorphic image of $L$ is either a prime ring or contains
only finitely many minimal prime ideals.
\end{proposition}

\begin{proof}
Assume that every homomorphic image of $L$ is a prime ring or contains only
finite number of minimal prime ideals. In view of Theorem
\ref{gr(I) product of primes => I product of primes}, we need only to show
that every graded ideal is a product/intersection of finitely many prime
ideals. Suppose $I=I(H,S)$ is a graded ideal of $L$. If $\bar{L}=L/I$ is a
prime ring, then $I$ is a prime ideal and we are done. Suppose, $I$ is not a
prime ideal. Now $\bar{L}=L/I\cong L_{K}(E\backslash(H,S))$ and by hypothesis,
contains only finitely many minimal prime ideals $P_{1},\cdot\cdot\cdot,P_{n}$
which clearly must all be graded as $gr(P_{j})$ is also a prime ideal for each
$j$ (Theorem \ref{R-2}(b)). We claim that their intersection must be zero.
Suppose, on the contrary, $A=P_{1}\cap\cdot\cdot\cdot\cap P_{n}\neq0$. Now $A$
is a non-zero graded ideal of $L_{K}(E\backslash(H,S))$ and so $A$ contains a
vertex $u\in E\backslash(H,S)$. Let $P$ be an ideal of $\bar{L}$ maximal with
respect to the property that $u\notin P$. We claim that $P$ is a prime ideal.
To see this, suppose $a\notin P$ and $b\notin P$ are elements of $\bar{L}$. So
$u\in\bar{L}a\bar{L}+P$ and $u\in\bar{L}b\bar{L}+P$. Then $u=u^{2}\in(\bar
{L}a\bar{L}+P)(\bar{L}b\bar{L}+P)=\bar{L}a\bar{L}b\bar{L}+P$. Since $u\notin
P$, this implies $a\bar{L}b\notin P$. Hence $P$ is a prime ideal of $\bar
{L}\footnote{I thank Zak Mesyan for pointing out this argument}$. As $P$ must
contain one of the minimal prime ideals $P_i$, we have $u\in P_i\subseteq P$,
a contradiction. Thus $P_1\cap\cdot\cdot\cdot\cap P_n=0$ and we conclude that
$I$ is the intersection and hence a product of the pre-images of the
$P_1,\cdot\cdot\cdot,P_n$ in $L$ all of which are prime ideals in $L$.

Conversely, suppose every ideal $I$ of $L$ is a product of prime ideals.
Consider the factor ring $L/I$. If $I$ is a prime ideal, then $L/I$ is a prime
ring. Otherwise, by hypothesis, $I=$ $P_{1}\cdot\cdot\cdot P_{n}$ where each
$P_{j}$ is a prime ideal. Then, in $L/I$, $\bar{P}_{1}\cdot\cdot\cdot\bar
{P}_{n}=0$, where $\bar{P}_{j}=P_{j}/I$. If $Q$ is a minimal prime ideal in
$L/I$, then $\bar{P}_{1}\cdot\cdot\cdot\bar{P}_{n}=0\in Q$ implies that, for
some $j$, $\bar{P}_{j}\subseteq Q$. By minimality, $Q=\bar{P}_{j}$. This shows
that $L/I$ contains only finitely many minimal prime ideals.
\end{proof}

\section{Examples}

We next construct various graphs illustrating the results obtained in the
preceding sections. These examples are also used to examine whether some of
the well-known theorems in commutative rings, such as the Cohen's theorem on
prime ideals \ and theorems on ZPI rings, hold for Leavitt path algebras.

\begin{example}
\label{N-Lttice graph} Consider the following \textquotedblleft$%
\mathbb{N}
\times%
\mathbb{N}
$-Lattice\textquotedblright\ graph $E$ where the vertices in $E$ are points in
the first quadrant of the coordinate plane whose coordinates are integers
$\geq0$. Specifically, $E^{0}=\{(m,n):m,n\in%
\mathbb{Z}
$ with $m,n\geq0\}$. Every vertex $(m,n)$ emits two edges connecting $(m,n)$
with $(m+1,n)$ and $(m,n+1)$.%

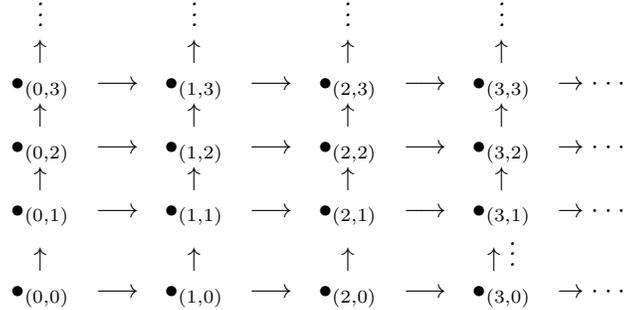
\[%
\begin{array}
[c]{cccccccc}%
\vdots &  & \vdots &  & \vdots &  & \vdots & \\
\uparrow &  & \uparrow &  & \uparrow &  & \uparrow & \\
\bullet_{(0,3)} & \longrightarrow & \bullet_{(1,3)} & \longrightarrow &
\bullet_{(2,3)} & \longrightarrow & \bullet_{(3,3)} & \rightarrow\cdot
\cdot\cdot\\
\uparrow &  & \uparrow &  & \uparrow &  & \uparrow & \\
\bullet_{(0,2)} & \longrightarrow & \bullet_{(1,2)} & \longrightarrow &
\bullet_{(2,2)} & \longrightarrow & \bullet_{(3,2)} & \rightarrow\cdot
\cdot\cdot\\
\uparrow &  & \uparrow &  & \uparrow &  & \uparrow & \\
\bullet_{(0,1)} & \longrightarrow & \bullet_{(1,1)} & \longrightarrow &
\bullet_{(2,1)} & \longrightarrow & \bullet_{(3,1)} & \rightarrow\cdot
\cdot\cdot\\
\uparrow &  & \uparrow &  & \uparrow &  & \uparrow\vdots & \\
\bullet_{(0,0)} & \longrightarrow & \bullet_{(1,0)} & \longrightarrow &
\bullet_{(2,0)} & \longrightarrow & \bullet_{(3,0)} & \rightarrow\cdot
\cdot\cdot
\end{array}
\]

\end{example}

Now $E$ is a row finite graph and and contains no cycles. So the ideals of
$L=L_{K}(E)$ are generated by vertices in $E$. If $A$ is a non-zero proper
ideal of $L$, then $A$ is a principal ideal. To see this, suppose $(i,j)\in A$
such that $i+j=k$ is the smallest integer, then it is easy to check that $A$
contains the entire quadrant $Q_{ij}=\{(m,n):m\geq i,n\geq j\}$. It is also
easy to verify that in this case, no other vertex $(m,n)$ satisfying $m+n=k$
will be in $A$. Thus $A=Q_{ij}$ is the principal ideal generated by the single
vertex $(i,j)$. Moreover, if $(i,j)\neq(m,n)$ are different vertices, then the
ideals $<(i,j)>\neq<(m,n)>$. We list below the various properties of $L$:

(a) $L$ is a prime ring.

(b) All the ideals of $L$ are graded.

(c) Every non-zero proper ideal of $L$ is a principal ideal generated by a
single vertex.

(d) An ideal $P$ is a prime ideal of $L$ if and only $P$ is generated by a
vertex "on the axis ", that is, is generated by a vertex of the form $(0,m)$
or $(m,0)$.

(d) Every ideal of $L$ is either a prime ideal or is a product of two prime
ideals. Indeed, if $A=<(m,n)>$ with $m\neq0$ and $n\neq0$, then $A=PQ=P\cap
Q$, where $P=<(m,0)>$ and $Q=<(0,n)>$.

\begin{example}
\label{Pyramid graph} Consider the following row-finite graph $F$ in which,
for each of the infinitely many $i$, there are two loops at the vertex $v_{i}$
and this indicated by $\underset{v_{i}}{\circlearrowleft\bullet
\circlearrowleft}$.%
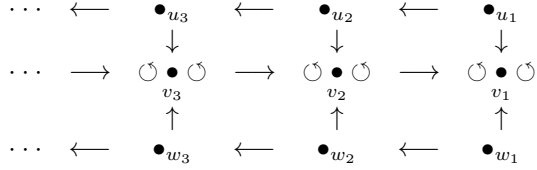
\[%
\begin{array}
[c]{cccccccc}
& \cdot\cdot\cdot & \longleftarrow & \bullet_{u_{3}} & \longleftarrow &
\bullet_{u_{2}} & \longleftarrow & \bullet_{u_{1}}\\
&  &  & \downarrow &  & \downarrow &  & \downarrow\\
& \cdot\cdot\cdot & \longrightarrow & \underset{v_{3}}{\circlearrowleft
\bullet\circlearrowleft} & \longrightarrow & \underset{v_{2}}{\circlearrowleft
\bullet\circlearrowleft} & \longrightarrow & \underset{v_{1}}{\circlearrowleft
\bullet\circlearrowleft}\\
&  &  & \uparrow &  & \uparrow &  & \uparrow\\
& \cdot\cdot\cdot & \longleftarrow & \bullet_{w_{3}} & \longleftarrow &
\bullet_{w_{2}} & \longleftarrow & \bullet_{w_{1}}%
\end{array}
\]
Now $F$ satisfies Condition (K) and so all the ideals of $L_{K}(F)$ are
graded. Also $F^{0}$ is downward directed and so $\{0\}$ is a prime ideal. For
each $n\geq1$, $H_{n}=\{v_{1},\cdot\cdot\cdot,v_{n}\}$ is a hereditary
saturated set and $E^{0}\backslash H_{n}$ is downward directed. Hence
$P_{n}=<H_{n}>$ is a prime ideal and we get an ascending chain of prime
ideals
\[
0\subset P_{1}\subset\cdot\cdot\cdot\subset P_{n}\subset\cdot\cdot\cdot
\qquad\qquad\qquad(\ast\ast).
\]
Let $P_{\omega}=%
{\displaystyle\bigcup\limits_{n\in\mathbb{N}}}
P_{n}$. Let $M_{1}=<u_{1}>$ and $M_{2}=<w_{1}>$ be the ideals generated by
$u_{1},w_{1}$ respectively in $L_{k}(F)$. It is straightforward to verify the following:

(a) The non-zero proper ideals of $L_{K}(F)$ are precisely the (graded) ideals
$P_{n}(n\geq1)$, $P_{\omega}$, $M_{1}$ and $M_{2}$.

(b) All the non-zero ideals of $L_{K}(F)$, other than $P_{\omega}$, are prime
ideals. The ideal $P_{\omega}$ is not a prime ideal since $E^{0}\backslash H$
is not downward directed, where $H=P_{\omega}\cap E^{0}=\{v_{n}:n\geq1\}$.
However, $P_{\omega}=M_{1}\cap M_{2}=M_{1}M_{2}$ is a product/intersection of
two prime ideals. Thus $L_{K}(F)$ is a prime ring in which every ideal is a
product of at most two prime ideals.

(c) For each $n$, $P_{n}=<v_{n}>$ is a principal ideal generated by the vertex
$v_{n}$. Thus all the non-zero ideals of $L_{K}(F)$, other than $P_{\omega}$,
are principal ideals. $P_{\omega}$ is not finitely generated.

(d) $L_{K}(F)$ is a two-sided artinian ring, but is not two-sided noetherian.
\end{example}

Example \ref{Pyramid graph} is also an example to illustrate the following statements.

I) A well-known theorem of Cohen states that if $R$ is a commutative ring and
if every prime ideal of $R$ is finitely generated, then $R$ is a noetherian
ring. Example \ref{Pyramid graph} shows that Cohen's theorem does not hold for
the two-sided ideals in Leavitt path algebras (as every prime ideal of
$L_{K}(F)$ is a principal ideal, but the ideal $P_{\omega}$ is not finitely
generated). Also the chain of ideals $(\ast\ast)$ shows that the ascending
chain condition does not hold in $L_{K}(F)$.

II) It is known (\cite{LM}) that if $R$ is a commutative ring in which every
ideal is a product of prime ideals, then $R$ must be a noetherian ring. Such
rings are known as generalized ZPI rings and have been completely
characterized. Example \ref{Pyramid graph} shows the Leavitt path algebra
$L_{K}(F)$ is a generalized ZPI ring, but it is not two-sided noetherian.

III) In a commutative ring, the union of an ascending chain of prime ideals is
again a prime ideal. In the Example \ref{Pyramid graph}, $P_{\omega}$ is the
union of a countable ascending chain of prime ideals, but $P_{\omega}$ is not
a prime ideal of $L_{K}(F)$.

IV) As a passing remark, we point out that in the case of a commutative ring
$R$ with identity, the ascending chain condition on ideals of $R$ is
equivalent to every ideal of $R$ being finitely generated. Example
\ref{Pyramid graph} shows that this no longer holds in Leavitt path algebras
as is clear by considering $L^{\prime}=L_{K}(F^{\prime})$, the Leavitt path
algebra of the graph $F^{\prime}=F\backslash\lbrack\{u_{n}:n\geq1\}\cup
\{w_{n}:n\geq1\}]$, where every proper ideal of $L^{\prime}$ is a principal
ideal, but the ascending chain condition for ideals does not hold in
$L^{\prime}$.

V) Also, in a commutative ring, a product of two finitely generated ideals is
again finitely generated. But in the Leavitt path algebra $L_{K}(F)$ in
Example \ref{Pyramid graph}, $M_{1}$ and $M_{2}$ are principal ideals, but the
product $M_{1}M_{2}=P_{\omega}$ is not finitely generated.

We next give an example of an ideal in a Leavitt path algebra which cannot be
factored as a product of finitely many prime ideals.

\begin{example}
Let $E$ be a graph with $E^{0}=\{v,v_{1},\cdot\cdot\cdot,v_{n},\cdot\cdot
\cdot\}$ and

$E^{1}=\{c_{1},\cdot\cdot\cdot,c_{n},\cdot\cdot\cdot\}\cup\{e_{1},\cdot
\cdot\cdot,e_{n},\cdot\cdot\cdot\}$. Further, for each $i$,

$s(e_{i})=$ $v_{i}$ and $r(e_{i})=v$. Also each $c_{i}$ is a loop at $v_{i}$,
$s(c_{i})=r(c_{i})=v_{i}$.

Now $L_{K}(E)$ is a prime ring. The prime ideals of $L_{K}(E)$ are $\{0\}$
and, for each $i=1,2,\cdot\cdot\cdot$, the graded ideal $P_{i}=<\{v_{j}:j\neq
i\}>$ and the non-graded ideal $Q_{i}^{p(x)}=P_{i}+<p(c_{i})>$, for each
irreducible polynomial $p(x)\in K[x]$, . Now the graded ideal $A=<v>$ is not
an intersection of finitely many prime ideals. Then, by Theorem
\ref{gr(I) product of primes => I product of primes}, we have, for any
irreducible polynomial $p(x)\in K[x]$ and for any subset $S\subseteq
\{c_{1},\cdot\cdot\cdot,c_{n},\cdot\cdot\cdot\}$, the ideal

$I=A+%
{\displaystyle\sum\limits_{c_{i}\in S}}
<\{p(c_{i})>$ is then not a product of finitely many prime ideals.
\end{example}

\begin{acknowledgement}
I thank Bruce Olberding for his helpful comments and S\"{o}ngul Esin and Muge
Kanuni for their careful reading of this paper and corrections.
\end{acknowledgement}

\end{document}